\documentclass[onefignum,onetabnum]{siamart190516}

\usepackage{enumitem}
\setlist[enumerate]{label=(\roman*)}
\usepackage{amsfonts}
\usepackage{amssymb}
\usepackage{mathrsfs}
\usepackage{dsfont}
\usepackage{mathtools}
\usepackage{algorithm}
\usepackage{algorithmic}
\usepackage{epstopdf}
\usepackage{subfigure} 

\renewcommand\d{\mathrm{d}}
\newcommand\N{\mathbb{N}}

\newcommand\R{\mathbb{R}}

\DeclareMathAlphabet{\mathpzc}{OT1}{pzc}{m}{it}
\newcommand\GG{\mathcal{G}}
\newcommand\LL{\mathscr{L}}
\newcommand\DD{\mathcal{D}}
\newcommand\MM{\mathcal{M}}
\newcommand\NN{\mathcal{N}}
\newcommand\PP{\mathcal{P}}
\newcommand\RR{\mathcal{R}}

\newcommand{\weakly}{\rightharpoonup}

\makeatletter
\newcommand{\oset}[3][0ex]{%
  \mathrel{\mathop{#3}\limits^{
    \vbox to#1{\kern-2\ex@
    \hbox{$\scriptstyle#2$}\vss}}}}
\makeatother

\newcommand{\compactly}{\oset{\textup{c}}{\hookrightarrow}}

\newcommand{\WOT}{\smash{\overset{\normalfont\textsc{wot}}{\to}}}

\DeclareMathOperator{\tr}{tr}

\DeclarePairedDelimiter\parens()

\DeclarePairedDelimiter\bracks[]
\DeclarePairedDelimiter\abs{\lvert}{\rvert}
\DeclarePairedDelimiter\norm{\lVert}{\rVert}
\DeclarePairedDelimiterX\innerp[2](){#1,#2}
\DeclarePairedDelimiterX\dual[2]{\langle}{\rangle}{#1,#2}
\providecommand\given{\nonscript\;\delimsize|\nonscript\;\mathopen{}}
\DeclarePairedDelimiterX\set[1]\{\}{#1}

\newsiamthm{assumption}{Assumption}
\crefname{assumption}{Assumption}{Assumptions}
\newsiamremark{remark}{Remark}
\newsiamthm{Algorithm}{Algorithm}

\crefname{ALC@unique}{Step}{Steps}

\headers{Semismoothness for Obstacle-Type VIs}
{Constantin Christof and Gerd Wachsmuth}

\title{Semismoothness for Solution Operators of 
Obstacle-Type Variational Inequalities with Applications in Optimal Control\thanks{Submitted to the editors DATE.
\funding{This research was supported by the German Research Foundation (DFG) under grant number WA 3636/4-2
within the priority program ``Non-smooth and Complementarity-based Distributed Parameter
Systems: Simulation and Hierarchical Optimization'' (SPP 1962).%
}}}

\author{Constantin Christof\thanks{%
Technische Universit\"at M\"unchen,
Faculty of Mathematics, M17,
85748 Garching bei M\"unchen, Germany,
\url{https://www-m17.ma.tum.de/Lehrstuhl/ConstantinChristof},
\email{christof@ma.tum.de}%
}%
\and
Gerd Wachsmuth\thanks{%
Brandenburgische Technische Universit\"at Cottbus-Senftenberg, 
Institute of Mathematics, 
03046 Cott\-bus, Germany, 
\url{https://www.b-tu.de/fg-optimale-steuerung},
\email{gerd.wachsmuth@b-tu.de},
ORCID: \href{https://orcid.org/0000-0002-3098-1503}{0000-0002-3098-1503}%
}%
}

\begin{document}

\maketitle

\begin{abstract}
We prove that solution operators of elliptic obstacle-type variational inequalities 
(or, more generally, locally Lipschitz continuous functions possessing certain pointwise-a.e.\ convexity properties)
are Newton differentiable when considered as maps between suitable Lebesgue spaces
and equipped with the strong-weak Bouligand differential 
as a generalized set-valued derivative. It is shown that this Newton differentiability allows
to solve optimal control problems with $H^1$-cost terms and one-sided pointwise control constraints by 
means of a semismooth Newton method. The superlinear convergence of the resulting algorithm 
is proved in the infinite-dimensional setting and its mesh independence 
is demonstrated in numerical experiments. We expect that the findings of this paper 
are also helpful for the design of numerical solution procedures for quasi-variational inequalities 
and the optimal control of obstacle-type variational problems. 

\end{abstract}




\section{Introduction and summary of results}
\label{sec:1}
Due to its importance for the analysis of 
generalized Newton methods
and the study of solution algorithms for nonsmooth optimization and optimal control problems, 
the concept of 
Newton differentiability (a.k.a.\ semismoothness) 
has received considerable attention in the literature. 
We refer to, e.g., \cite{ChenNashedQi2000,Facchinei2003,Hintermueller2002,Izmailov2014,Kummer1988,Qi1999,Schiela2008,Ulbrich2011},
which discuss semismooth Newton methods for 
equations in finite- and infinite-dimensional spaces,
and to \cite{Bihain1984,Mifflin1977,Mifflin1977-2,SchrammZowe1992}, which are concerned with the minimization of semismooth functions.
Despite this prominent role that the notion of Newton differentiability plays in the field of nonsmooth analysis and optimization, 
contributions which establish Newton differentiability properties for functions between infinite-dimensional spaces
that arise as control-to-state mappings 
or as parts of stationarity systems in optimal control applications are comparatively scarce. 
The result that is most commonly used in this field is 
the well-known fact that Nemytskii operators which are induced by a Lipschitz continuous, semismooth
function $f\colon \R \to \R$ are 
Newton differentiable as mappings between Lebesgue spaces in the presence of a norm gap.
See, for example, \cite{Reyes2005,Hintermueller2008,Roesch2011,Schiela2008,Ulbrich2011}, where this Newton differentiability property is 
exploited to set up semismooth Newton methods for control- and state-constrained optimal control problems, 
and \cite{Hintermueller2016,Kunisch2021}, 
where a similar approach is used for the analysis of regularized variational inequalities. 
For operators 
that are not of Nemytskii type (and in which superposition operators are not the sole source of nonsmoothness, 
see, e.g., the solution map of the partial differential equation considered in \cite{Christof2018})
much less is known. One of the few contributions that accomplishes to prove Newton differentiability properties for 
a nontrivial example of such a function is \cite{Brokate2020}, 
which establishes the Newton differentiability of the scalar play and stop operator 
(and thus of  the solution map  of a prototypical rate-independent evolution variational inequality) by means 
of an explicit solution formula involving the accumulated maximum.
In \cite{Brokate2019}, the findings of \cite{Brokate2020} are extended 
to a parabolic PDE system involving the scalar play.
The nonexistence
of further results and of a comprehensive theory on the Newton differentiability of nonsmooth operators arising in the field of
optimal control is rather unsatisfactory---in particular in view of the multitude of contributions 
on the semismoothness of functions in the finite-dimensional setting, see, e.g., \cite{Bolte2009,OutrataBook1998,Qi1999}
and the references therein.  

The aim of this paper is to demonstrate that, beside superposition operators 
and the scalar play and stop considered in \cite{Brokate2020,Brokate2019}, there is a further large class of operators
arising in optimal control applications that are Newton differentiable when endowed with a 
suitable (and computable) set-valued derivative, namely, 
solution mappings of obstacle-type variational inequalities (VIs) with unilateral constraints.
Such functions  arise, for instance, when optimal control problems governed by partial differential equations (PDEs) 
with $H^1$-controls are studied, see \cref{sec:5,sec:6}, or in the field of optimal control of contact problems,
see \cite{HertleinRaulsUlbrichUlbrich2022,Kunisch2021}. 
The main idea of our analysis is to exploit that solution maps of obstacle-type VIs
possess pointwise-a.e.\ convexity properties which, in combination with certain compact embeddings, 
immediately yield Newton differentiability results when the strong-weak Bouligand differential 
is used as a generalized set-valued derivative, see \cref{def:Bsubdifferential,th:mainabstractsemismooth} below. Along these lines,
one obtains that Newton differentiability is readily available 
for solution operators of VIs like the classical obstacle problem or the scalar Signorini problem
 when these functions are considered as maps between suitable 
Lebesgue spaces. 
The content of this paper can be summarized as follows:

In \cref{sec:2}, we discuss Newton differentiability properties of pointwise-a.e.\ convex operators
on a general abstract level. Here, we prove that such functions are indeed 
Newton differentiable when endowed with the 
strong-weak Bouligand differential 
and considered as functions between suitable Lebesgue spaces. For the main result of this section, 
we refer the reader to \cref{th:mainabstractsemismooth}.

In \cref{sec:3}, we illustrate that the abstract results of \cref{sec:2} can be applied 
to the solution operators of obstacle-type VIs. During the course of the
analysis of this section, we also generalize well-known truncation arguments of Stampacchia, 
see \cref{lemma:truncation,cor:VIsemismooth1}
for the main results on this topic. 

\Cref{sec:4} contains two tangible examples of obstacle-type VIs that are covered by our analysis:
the scalar Signorini problem and the classical obstacle problem. Here, we also recall a recent 
characterization result for the strong-weak Bouligand differential of the solution map
of the classical obstacle problem which, in combination with the analysis of \cref{sec:3},
provides a framework that can be readily used for the design of semismooth Newton methods 
or comparable algorithms
in practical applications. 

In \cref{sec:5}, we consider an example of such an application, namely,
the numerical solution of an optimal control problem 
with unilateral control constraints. For this problem, we design a 
semismooth Newton method in function space and establish its local superlinear convergence
in infinite dimensions, see \cref{th:NewtonConvergence}. 

\Cref{sec:6} concludes the paper with numerical experiments
which illustrate that the algorithm developed in \cref{sec:5} 
indeed converges superlinearly and, due to the established convergence in the function space setting,
mesh-independently. This section also contains some comments on further applications of our results,
e.g., in the context of optimal control problems governed by 
obstacle-type VIs and the field of quasi-variational inequalities.

Before we begin with our analysis, 
we would like to point out that 
techniques very similar to those used in \cref{sec:2} of this paper have recently also been employed in 
\cite{BrokateUlbrich2022} for the study of parametric semi\-smooth functions,
see \cite[sections~3 and 4]{BrokateUlbrich2022}. The main difference between our approach and that of
\cite{BrokateUlbrich2022} is that our analysis is tailored to applications 
in the field of optimal control and the area of obstacle-type VIs.
This is in particular emphasized by the 
nested Banach space structure that we consider (see \cref{ass:functionalanalytic,ass:VIs})
and the generalized differential that we work with---the 
already mentioned strong-weak 
Bouligand differential, see \cref{def:Bsubdifferential}. 
This differential arises on the operator-theoretic level as an immediate consequence of Rademacher's theorem
when locally Lipschitz continuous 
control-to-state mappings are considered, see \cref{th:infRademacher},
and possesses several advantageous properties (e.g., regarding chain rules and 
adjoint-based approaches) that make it
the appropriate choice for many optimal control applications, see \cref{sec:5,sec:6}. 
Moreover, for several solution operators (e.g., those of nonsmooth semilinear PDEs, the classical obstacle problem, 
and the bilateral obstacle problem), formulas for certain elements of the 
strong-weak Bouligand differential or even full characterization results 
have recently been obtained in the literature, see \cite{Christof2018,RaulsUlbrich2019,RaulsUlbrich2021,Rauls2020}.
Generalized derivatives of strong-weak Bouligand type are thus readily available 
and can, in combination with the semismoothness results established in 
\cref{th:mainabstractsemismooth,cor:VIsemismooth1} of the present paper,
immediately be used for setting up numerical solution algorithms for optimal control problems, cf.\
the semismooth Newton method developed in \cref{sec:5}. The generalized differential 
studied in \cite[section~4, Equation~4.12]{BrokateUlbrich2022}, which relies on pointwise measurable selections, 
is less tangible in this regard---in particular in the context of obstacle-type variational problems. 
We remark that, by restricting the attention 
to Bouligand generalized derivatives, we are also able to completely avoid
working with measurable selectors and the assumptions that they require, cf.\
\cite[sections~4 and 5]{BrokateUlbrich2022} and the proof of \cref{th:mainabstractsemismooth}. 
This allows us in particular to also prove the Newton differentiability of, 
for instance,
the solution operator $S\colon L^2(\Omega) \to L^2(\Omega)$, $u \mapsto y$,
of the classical obstacle problem in situations in which the 
functions $S(u) \in L^2(\Omega)$ do not possess continuous representatives and in which,
as a consequence, 
the Carathéodory conditions or local pointwise Lipschitz estimates cannot be satisfied by the function
$L^2(\Omega) \times \Omega \ni (u, \omega) \mapsto S(u)(\omega) \in \R$, 
see \cite[Equations~4.6, 4.7, 4.22]{BrokateUlbrich2022}, \cref{subsec:obstacle},
and \cref{th:mainabstractsemismooth}.
The downside of our approach in comparison with that of \cite{BrokateUlbrich2022} is, 
of course, that it only applies to pointwise-a.e.\ convex operators, cf.\ \cite[section~4]{BrokateUlbrich2022}.

\subsection{Remarks on the notation}%
We use the symbols $\|\cdot\|$, $(\cdot, \cdot)$, and $\langle\cdot,\cdot\rangle$
to denote norms, scalar products, and dual pairings, respectively, 
with a subscript indicating which spaces this notation is referring to.
Strong and weak convergence are denoted by the arrows $\to$ and $\weakly$, respectively.
Given two normed spaces
$X$
and
$Y$
satisfying $X \subset Y$, we write $X \hookrightarrow Y$
if $X$ is continuously embedded into $Y$, i.e., if the inclusion map $\iota: X \to Y$, $x \mapsto x$,
is a linear and
continuous function. If the inclusion map is even compact, then we say that $X$ is 
compactly embedded into $Y$ and write $X \compactly Y$. 
With
$\LL(X,Y)$,
we denote the space of all linear and continuous functions
on a normed space
$X$
with values in
$Y$.
In the special case
$Y = \R$,
$X^* := \LL(X,\R)$ denotes the topological dual of $X$.
Given a sequence $\{G_n\} \subset \LL(X,Y)$, we say that 
$G_n$ converges in the weak operator topology (WOT) to $G \in \LL(X,Y)$,
in symbols $\smash{G_n \WOT G}$, if $G_n z \weakly Gz$ holds in $Y$ for all $z \in X$. 
In addition to these conventions, new notation is introduced in the following sections 
wherever necessary.
These symbols are explained upon their first appearance.


\section{Newton differentiability of pointwise-a.e.\ convex operators}%
\label{sec:2}%
In this section, we prove general Newton differentiability results for maps that
are (in an appropriately defined sense) locally Lipschitz continuous and pointwise-a.e.\ convex. 
The setting that we consider for our analysis is as follows.

\begin{assumption}[Standing assumptions for the analysis of \cref{sec:2}]%
\label{ass:functionalanalytic}%
Throughout this section, we assume the following:
\begin{enumerate}
\item\label{ass:functionalanalytic:i} 
$(\Omega, \Sigma, \mu)$ is a complete measure space with associated real
Lebesgue spaces $(L^p(\Omega), \|\cdot\|_{L^p(\Omega)})$, $p \in [1, \infty]$.
\item\label{ass:functionalanalytic:ii} 
 $(Y, \|\cdot\|_Y)$ is a real separable reflexive 
Banach space such that $Y \compactly L^q(\Omega)$ holds
 for a fixed
$q \in [1,\infty]$.
\item\label{ass:functionalanalytic:iii} 
$(X, \|\cdot\|_X)$ is a real separable Banach space.
\item\label{ass:functionalanalytic:iv}  
$(U, \|\cdot\|_U)$ is a real reflexive Banach space satisfying 
 $U \compactly X$.
\item\label{ass:functionalanalytic:v} 
The map $S\colon X \to Y$ satisfies 
\begin{equation}
\label{eq:convex}
S(\lambda x_1 + (1 - \lambda)x_2) \leq \lambda S(x_1) + (1-\lambda)S(x_2)\quad\mu\text{-a.e.\ in }\Omega
\end{equation}
for all $x_1, x_2 \in X$ and all $\lambda \in [0,1]$. Further, $S$ is locally 
Lipschitz continuous in the following sense:
There exists an exponent $r \in [1, \infty]$ such that, 
for all $x \in X$, there exist constants $C, \varepsilon >  0$ satisfying 
\begin{equation}
\label{eq:randomeq2625}
\norm*{ S(x_1) - S(x_2) }_Y \leq C  \|x_1 - x_2\|_X
\end{equation}
for all $x_1, x_2 \in X$ with $\|x_i - x\|_X \leq \varepsilon$, $i=1,2$, 
and 
\begin{equation}
\label{eq:randomeq2625-25}
\norm*{ S(x_1 + z) - S(x_1)
}_{L^r(\Omega)} \leq C \|z\|_U
\end{equation}
for all $x_1 \in X$ and all $z \in U$ with $\|x_1 - x\|_X \leq \varepsilon $ and $\|z\|_U \leq \varepsilon $.
\end{enumerate}
\end{assumption}

Note that condition \eqref{eq:randomeq2625-25} in \cref{ass:functionalanalytic}\ref{ass:functionalanalytic:v} 
is always satisfied for $r=q$ by the local Lipschitz continuity of $S$ as a function from $X$ to $Y$ 
in \eqref{eq:randomeq2625}
and the
continuity of the embeddings $Y \hookrightarrow L^q(\Omega)$ and $U \hookrightarrow X$. 
The modified stability estimate \eqref{eq:randomeq2625-25}
allows to establish the Newton differentiability of $S$ in stronger spaces if a better Lipschitz estimate 
for perturbations $z$ from the space $U$ is available for $S$, see \cref{th:mainabstractsemismooth} below and 
the tangible examples in 
\cref{sec:4}. If this  is not the case, then with the trivial choice $r = q$ 
\cref{ass:functionalanalytic}\ref{ass:functionalanalytic:v}  boils down to the condition 
that $S$ should be locally Lipschitz as a function from $X$ to $Y$ and pointwise-a.e.\ 
convex in the sense of \eqref{eq:convex}. 
Regarding the separability of $Y$ in \ref{ass:functionalanalytic:ii}, we would like to point out that 
this assumption can be made without any loss of generality. 
Indeed, if $Y$ is not separable,
then one can simply replace this space by the (necessarily separable) closure of the linear hull of the image $S(X)$ in $Y$
and thus resort to the separable situation.
Next, we recall
the definition of Gâteaux differentiability.

\begin{definition}[Gâteaux differentiability]%
\label{def:Gateaux}%
The function $S\colon X \to Y$ is called Gâteaux differentiable at a point  $x \in X$
if the directional derivative 
\[
S'(x;z) := \lim_{t \to 0^+} \frac{S(x + tz) - S(x)}{t} \in Y
\]
exists for all $z \in X$ and if the map $X \ni z \mapsto S'(x;z) \in Y$ is linear and continuous.
In this case, we call $S'(x):= S'(x;\cdot) \in \LL(X,Y)$ the Gâteaux derivative of $S$ at $x$.
\end{definition}

The properties of $S$, $X$, and $Y$
yield the next result.

\begin{theorem}
\label{th:infRademacher}
The set of points in $X$ at which the function $S\colon X \to Y$ possesses a Gâteaux derivative
is dense in $X$.
Henceforth, this set is denoted by $\DD_S$.
\end{theorem}
\begin{proof}
See \cite[Proposition~1.3, Remark~1.3]{Thibault1982} and the references therein. 
\end{proof}

To establish the Newton differentiability of $S$,
we use a generalized differential.

\begin{definition}[Strong-weak Bouligand differential]%
\label{def:Bsubdifferential}%
For all $x \in X$, we define the 
strong-weak Bouligand differential $\partial_B^{sw}S(x) \subset \LL(X,Y)$ by
\begin{equation*}
\begin{aligned}
\partial_B^{sw}S(x)
:=
\set[\big]{
G 
\given \exists \{x_n\} \subset \DD_S \colon x_n \to x \text{ in }X,~ S'(x_n) \WOT G \text{ in } \LL(X, Y)}.
\end{aligned}
\end{equation*}
\end{definition}

Note that the notation $\partial_B^{sw}S(x)$ emphasizes the modes of convergence appearing 
in the definition of the strong-weak Bouligand differential (strong convergence for the base points $x_n$ and 
WOT-convergence for the derivatives), cf.\ \cite[Definition~3.1]{Christof2018}.
Due to the separability of $X$ and the reflexivity of $Y$,
we have the following variant of the Banach--Alaoglu theorem.

\begin{theorem}%
\label{th:BanachAlaoglu}%
Every bounded sequence $\{G_n\} \subset \LL(X, Y)$ possesses a subsequence 
that converges w.r.t.\ the WOT in $\LL(X, Y)$ to an operator  $G \in \LL(X, Y)$.
\end{theorem}
\begin{proof}
	We set $C := \sup_{n \in \N} \norm{G_n}_{\LL(X,Y)} < \infty$.
	Let $\{x_k\}_{k = 1}^\infty \subset X$ be dense.
	Since $Y$ is reflexive, the sequence $\{ G_n x_k\}_{n=1}^\infty$
	possesses a weak accumulation point for every $k \in \N$.
	By a standard diagonal argument, we can pick a subsequence
	$\{\hat G_n\}$ of $\{G_n\}$ such that
	$\hat G_n x_k \weakly g_k$ holds in $Y$ for $n \to \infty$ for all $k \in \N$ with some $g_k \in Y$.
	For an arbitrary $x \in X$, there further exists a sequence $\{x_{k_m}\}$
	with $x_{k_m} \to x$.
	From
	\begin{equation*}
		\norm{
			g_{k_m}
			-
			g_{k_l}
		}_Y
		\le
		\liminf_{n \to \infty}
		\norm{
			\hat G_n x_{k_m}
			-
			\hat G_n x_{k_l}
		}_Y
		\le
		C \norm{ x_{k_m} - x_{k_l} }_X,
	\end{equation*}
	we obtain that $\{g_{k_m}\}$ is Cauchy and thus convergent.
	It is easy to check that the limit only depends on $x$
	(and not on $\{x_{k_m}\}$). This allows us to define $G x := \lim_{m \to \infty} g_{k_m}$.
	The linearity of $x \mapsto G x$ is evident
	and the boundedness of $G$ follows from
	\begin{equation*}
		\norm{G x}_Y
		=
		\lim_{m \to \infty} \norm{g_{k_m}}_Y
		\le
		\limsup_{m \to \infty} \liminf_{n \to \infty} \norm{\hat G_n x_{k_m}}_Y
		\le
		C \norm{x}_X
		\qquad\forall x \in X.
	\end{equation*}
	It remains to show that $\hat G_n \WOT G$ holds.
	For arbitrary $x \in X$ and $y^* \in Y^*$, we have
	\begin{equation*}
		\abs{\dual{y^*}{(\hat G_n - G) x}_Y}
		\le
		\abs{\dual{y^*}{(\hat G_n - G) x_k}_Y}
		+
		2 C \norm{y^*}_{Y^*} \norm{x - x_k}_X\qquad \forall k \in \N.
	\end{equation*}
	Since $\{x_k\} \subset X$ is dense in $X$
	and since $\hat G_n x_k \weakly g_k = G x_k$ holds for $n \to \infty$ for all fixed $k$,
	this implies
	$\hat G_n \WOT G$ as claimed.
\end{proof}

As a consequence,
we obtain the following result (see \cite[Proposition~2.1]{Thibault1982}).

\begin{corollary}%
\label{cor:Bouligandnonempty}%
The generalized differential $\partial_B^{sw}S(x)$ is nonempty for all $x \in X$.
\end{corollary}
\begin{proof}
Given $x \in X$, we can find a sequence $\{x_n\} \subset \DD_S$ with $x_n \to x$ by \cref{th:infRademacher}.
Due to the local Lipschitz continuity of $S\colon X \to Y$, the 
sequence of Gâteaux derivatives $S'(x_n)$ is bounded in $\LL(X, Y)$. 
There thus exists a subsequence of $\{S'(x_n)\}$
(still denoted the same) such that $S'(x_n) \WOT G$ holds
in $\LL(X, Y)$ for some $G \in \LL(X, Y)$. 
By \cref{def:Bsubdifferential},
this $G$ satisfies $G \in \partial_B^{sw}S(x)$.
\end{proof}

Using standard techniques, we can also prove the following 
upper semicontinuity result for the set function 
$\partial_B^{sw} S\colon X \rightrightarrows \LL(X,Y)$.

\begin{lemma}
\label{lemma:uppersemi}
Let $\{x_n\} \subset X$ and $\{G_n\} \subset \LL(X, Y)$ 
be sequences satisfying 
$x_n\to x$ in $X$ for some $x \in X$, 
$G_n \in \partial_B^{sw} S(x_n)$ for all $n \in \mathbb{N}$,
and $G_n \to G$ w.r.t.\ the WOT in $\LL(X, Y)$ for some 
$G \in \LL(X, Y)$. Then it holds $G \in \partial_B^{sw} S(x)$. 
\end{lemma}
\begin{proof}
The proof of this result is completely analogous to that of \cite[Proposition~2.11(iii)]{Rauls2020},
see also \cite[Proposition~3.6]{Christof2018}.
Note that in \cite{Rauls2020},
$S$
is assumed to be 
globally Lipschitz continuous,
but
local Lipschitz continuity suffices for the proof.
\end{proof}

We prepare
our Newton differentiability result for $S$
with three lemmas.

\begin{lemma}
\label{lemma:convex1}
For all $x \in X$ and all $G \in \partial_B^{sw}S(x)$, we have 
\[
S( x + z) - S(x) \geq Gz~\mu\text{-a.e.\ in }\Omega \qquad \forall z \in X.
\]
\end{lemma}
\begin{proof}
Let $\{x_n\} \subset \DD_S$ be an approximating sequence for $G \in \partial_B^{sw}S(x)$
as in the definition of the generalized differential $\partial_B^{sw}S(x)$. Then, for each $n$, we obtain from the pointwise-a.e.\
convexity of $S$, the Gâteaux differentiability of $S\colon X \to Y$ in $x_n$,
and the embedding $Y \hookrightarrow L^q(\Omega)$ that 
\[
S( x_n + z) - S(x_n)
\geq \lim_{\mathbb{N} \ni k \to \infty} \frac{S( x_n + (1/k) z) - S(x_n) }{1/k} =
 S'(x_n)z~\mu\text{-a.e.\ in }\Omega \quad \forall z \in X.
\]
Passing to the limit $n \to \infty$ in this inequality by using the local Lipschitz continuity of $S$
and again the embedding  $Y \hookrightarrow L^q(\Omega)$ yields the claim.
\end{proof}

\begin{lemma}
	\label{lem:closed_balls}
	For every $R \ge 0$, the set
	\(
		\set[\big]{
			w \in L^q(\Omega) \cap L^r(\Omega)
			\given
			\norm{w}_{L^r(\Omega)} \le R
		}
	\)
	(with the exponents $q$ and $r$ from \cref{ass:functionalanalytic})
	is sequentially weakly closed in $L^q(\Omega)$.
\end{lemma}
\begin{proof}
	One can check that
	$L^q(\Omega) \ni w \mapsto \norm{w}_{L^r(\Omega)} \in [0, \infty]$
	is convex and lower semicontinuous.
	Thus it is weakly lower semicontinuous and the claim follows.
\end{proof}

\begin{lemma}
\label{lemma:bounded1}
For every $x \in X$, there exist constants $C, \delta > 0$ such that
\[
\sup_{v \in X, \|v - x\|_X \leq \delta} \,\sup_{G \in \partial_B^{sw}S(v)} \|G z\|_{L^r(\Omega)} \leq C \|z\|_U
\qquad\forall z \in U.
\]
Here, $r \in [1,\infty]$ denotes the exponent from \cref{ass:functionalanalytic}\ref{ass:functionalanalytic:v}.
\end{lemma}

\begin{proof}
Suppose that $x \in X$ is given and let 
$C, \varepsilon > 0$ be the constants from 
\cref{ass:functionalanalytic}\ref{ass:functionalanalytic:v} for this $x$.
We set $\delta := \varepsilon/2$.
Let
$v \in X$ with $\|v - x\|_X \leq \delta$
and $G \in \partial_B^{sw}S(v)$ be arbitrary,
and let $\{v_n\} \subset X$ be an approximating sequence of Gâteaux points 
for $G$ as in \cref{def:Bsubdifferential}. We assume
w.l.o.g.\ that $\|v_n - x\|_X \leq \varepsilon$ holds for all $n \in \N$.
Due to \eqref{eq:randomeq2625-25}, this yields
\begin{equation}
\label{eq:randomeq2763263}
\norm*{ \frac{S(v_n + t z) - S(v_n)}{t}
}_{L^r(\Omega)} \leq C\|z\|_U
\qquad
\forall n \in \N, t \in \parens*{0, \frac {\varepsilon}{\norm{z}_U}}, z \in U,
\end{equation}
where the fraction $\varepsilon/\norm{z}_U$ is to be understood as $\infty$ in the case $z=0$.
Because of the Gâteaux differentiability of $S\colon X \to Y$ at $v_n$
and since $Y \hookrightarrow L^q(\Omega)$,
we can use \cref{lem:closed_balls}
to pass to the limit $t \to 0^+$ in \eqref{eq:randomeq2763263}
and obtain
\[
\norm*{ S'(v_n)z 
}_{L^r(\Omega)} \leq C \|z\|_U\qquad \forall n \in \N, z \in U.
\]
By $S'(v_n) \WOT G$ in $\LL(X, Y)$
and $Y \hookrightarrow L^q(\Omega)$,
we have $S'(v_n)z  \weakly Gz$ in $L^q(\Omega)$. 
Using \cref{lem:closed_balls} again
gives
$
\norm*{ G z
}_{L^r(\Omega)} \leq C \|z\|_U
$
for all $z \in U$
and the claim follows.
\end{proof}

With \cref{lemma:convex1,lemma:bounded1} at hand, we are 
in the position to prove our first main result.
Before we do so, we clarify what we mean with the term ``Newton differentiable''
in the situation of the nested Banach space structure in \cref{ass:functionalanalytic}.

\begin{definition}[Newton differentiability]%
\label{def:NewtDif}%
Suppose that $\GG\colon X \rightrightarrows \LL(X,Y)$ is a set-valued map
and that $p \in [1, \infty]$ is an exponent. 
We say that $S\colon X \to Y$ with $\GG$ is Newton differentiable 
(with Newton derivative $\GG$)
w.r.t.\ perturbations in $U$
with values in $L^p(\Omega)$
if every $x \in X$ satisfies
\[
\sup_{G \in \GG(x + z)} \frac{\|S(x + z) - S(x) - Gz\|_{L^p(\Omega)}}{\|z\|_{U}} \to 0
\quad\text{for } \|z\|_{U} \to 0. 
\]
\end{definition}

Note that \cref{def:NewtDif} allows to distinguish between 
different regularities of the points $x$ and the perturbations $z$, cf.\ \cite[Definition~3.1]{Ulbrich2011}.

\begin{theorem}[Newton differentiability of $S$]%
\label{th:mainabstractsemismooth}%
Let $r,q \in [1, \infty]$ be the exponents from \cref{ass:functionalanalytic}
and let $p \in \{q\} \cup (\min(q,r), \max(q,r))$ be given.
Then the function $S\colon X \to Y$ with the differential $\partial_B^{sw} S\colon X \rightrightarrows \LL(X,Y)$
is Newton differentiable w.r.t.\ perturbations in $U$
with values in $L^p(\Omega)$.
\end{theorem}

\begin{proof}
	Let $x \in X$ be fixed.
	Due to \cref{cor:Bouligandnonempty}, 
	it suffices to show that, for all $\{z_n\} \subset U \setminus \set{0}$,
	$\{G_n\} \subset \LL(X,Y)$
	with $\|z_n\|_{U} \to 0$,
	$G_n \in \partial_B^{sw}S(x + z_n)$,
	we have
	\begin{equation}
		\label{eq:randomeq2635}
		\limsup_{n \to \infty} \frac{\norm{S(x + z_n) - S(x) - G_n z_n}_{L^p(\Omega)}}{\norm{z_n}_{U}}
		=
		0
		.
	\end{equation}
	Let such sequences $\{z_n\}$ and $\{G_n\}$ be given.
	To prove \eqref{eq:randomeq2635}, we pass over to subsequences of $\{z_n\}$ and $\{G_n\}$
	(still denoted the same) along which the limit superior in 
	\eqref{eq:randomeq2635} is attained as a limit. 
	Since \eqref{eq:randomeq2625} and the embedding $U \hookrightarrow X$ imply that 
	$\{G_n\}$ is bounded in $\LL(X, Y)$,
      we may assume w.l.o.g.\ that the sequence $\{G_n\}$
	satisfies $G_n \WOT G$ in $\LL(X, Y)$
	for some $G \in \partial_B^{sw}S(x)$, see \cref{th:BanachAlaoglu} and \cref{lemma:uppersemi}.
	Due to the reflexivity of $U$  and $U \compactly X$,
	we further assume w.l.o.g.\ that the sequence $e_n := z_n/\|z_n\|_{U}$  
	converges weakly in $U$ and strongly in $X$ to some $e \in U$.
	From \cref{lemma:convex1},
	$G \in \partial_B^{sw}S(x)$,
	and
	$G_n \in \partial_B^{sw}S(x + z_n)$,
	we now get
	\[
		S(x) - S(x + z_n) \geq - G_n z_n
		\quad\text{and}\quad
		S(x + z_n) - S(x) \geq   G   z_n
		\qquad\mu\text{-a.e.\ in }\Omega,
	\]
and, as a consequence,
	\begin{equation}
	\label{eq:randomeq27636gd636354}
		0
		\geq  \frac{ S(x + z_n) - S(x) - G_n z_n}{\|z_n\|_{U}} 
		\geq  \frac{ G z_n - G_n z_n}{\| z_n\|_{U}} 
		=
		(G - G_n) e_n
		\qquad \mu\text{-a.e.\ in }\Omega. 
	\end{equation}
	Integrating (or taking the essential supremum in the case $p=\infty$) in \eqref{eq:randomeq27636gd636354}
	gives 
	\begin{equation}
	\label{eq:randomeq27636gd63}
		\frac{\|S(x + z_n) - S(x) - G_n z_n \|_{L^p(\Omega)}}{\|z_n\|_{U}} 
		\leq
		\|(G- G_n)e_n\|_{L^p(\Omega)}.
	\end{equation}
	Note 
	that the choice of $p$ 
	and Hölder's inequality imply the existence of $\alpha \in (0,1]$ 
	with
	$	\norm*{
		v
		}_{L^p(\Omega)}
		\leq
		\norm*{
		v
		}_{L^q(\Omega)}^\alpha
		\norm*{
		v
		}_{L^r(\Omega)}^{1-\alpha}
	$
	for all $v \in L^q(\Omega) \cap L^r(\Omega)$.
	We thus have 
	\begin{equation}
	\label{eq:randomestimate435-1}
		\norm*{
			(G - G_n)e_n
			}_{L^p(\Omega)}
			\leq
			\norm*{
			(G - G_n)e_n
			}_{L^q(\Omega)}^{\alpha}
			\norm*{
			(G - G_n)e_n
			}_{L^r(\Omega)}^{1 - \alpha}.
	\end{equation}
	Due to \cref{lemma:bounded1}, the embedding $U \compactly X$, 
	the convergence $\|z_n\|_U \to 0$, and the identity $\|e_n\|_U = 1$, we know that there 
	exists a constant $C>0$ satisfying 
	$\norm*{
			(G - G_n)e_n
			}_{L^r(\Omega)}^{1 - \alpha} \leq C$.
	If we use this bound in \eqref{eq:randomestimate435-1} and employ the triangle inequality, 
	then it follows that 
		\begin{equation}
	\label{eq:randomestimate435-2}
		\norm*{
			(G - G_n)e_n
			}_{L^p(\Omega)}
			\leq C
			\parens*{
				\norm*{
				(G - G_n)e
				}_{L^q(\Omega)}
				+
				\norm*{
				(G - G_n)(e- e_n)
				}_{L^q(\Omega)}
			}^{\alpha}
	\end{equation}
	holds for all $n$. 
	Since $\{G_n\}$ is bounded in $\LL(X, Y)$ and since $Y$ is compactly embedded into $L^q(\Omega)$,
	\eqref{eq:randomestimate435-2} yields that, for a potentially larger constant $C$, we have 
			\begin{equation}
	\label{eq:randomestimate435-3}
		\norm*{
			(G - G_n)e_n
			}_{L^p(\Omega)}
			\leq C
			\parens*{
				\norm*{
				(G - G_n)e
				}_{L^q(\Omega)}
				+
				\norm*{
				e- e_n
				}_{X}
			}^{\alpha}.
	\end{equation}
	As $e_n \to e$ holds in $X$  and since the embedding $Y \compactly L^q(\Omega)$ and the 
	convergence $(G - G_n)e \weakly 0$ in $Y$ imply that $\norm*{
				(G - G_n)e
				}_{L^q(\Omega)} \to 0$ holds, 
				we obtain from \eqref{eq:randomestimate435-3} that 
				the norm $\norm*{
			(G - G_n)e_n
			}_{L^p(\Omega)}$ converges to zero. 
			In combination with \eqref{eq:randomeq27636gd63},
			this yields \eqref{eq:randomeq2635} and completes the proof.
\end{proof}

Note that the last result remains valid
 when the differential $\partial_B^{sw} S(x)$ is replaced by 
 the WOT-closure in $\LL(X,Y)$ of the convex hull of $\partial_B^{sw} S(x)$ (as one may easily check). 
 In applications, in which a convex Newton derivative is desirable, 
 this can thus always be achieved by enlarging the strong-weak Bouligand differential in \cref{th:mainabstractsemismooth}. 

\section{Application to obstacle-type VIs}
\label{sec:3}

In this section, we show that the results of \cref{sec:2} can be applied to 
solution maps 
of obstacle-type VIs of the form
\begin{equation*}
\label{eq:V}
\tag{\textup{VI}}
y \in K,\qquad 
\dual*{ Ay + f(y) - u}{ v - y }_V \geq 0\qquad \forall v \in K. 
\end{equation*}
Our standing assumptions are as follows:

\begin{assumption}[Standing assumptions for the analysis of \cref{sec:3}]%
\label{ass:VIs}%
Throughout this section, we assume the following (unless explicitly stated otherwise):
\begin{enumerate}
\item\label{VIass:item:i}
 $(\Omega, \Sigma, \mu)$ is a complete and finite measure space with associated real
Lebesgue spaces $(L^p(\Omega), \|\cdot\|_{L^p(\Omega)})$, $1 \leq p \leq \infty$.
\item\label{VIass:item:ii}
 $(V, \|\cdot\|_V)$ is a real separable Hilbert space such that
 $V \compactly L^q(\Omega)$
 is dense
 for a fixed $q \in [2,\infty]$.
Further, the truncations 
\[
\bracks*{ v }_{a_1}^{a_2} := \min\parens*{ a_2, \max \parens*{a_1, v} } 
\]
satisfy $\bracks*{ v }_{a_1}^{a_2} \in V$ for all $a_1, a_2 \in [-\infty, \infty]$
with $a_1 \leq 0 \leq a_2$ and all $v \in V$. 
Here, $\min(a_2, \cdot)$ and $\max(a_1, \cdot)$ act by superposition, i.e., $\mu$-a.e.\ in $\Omega$.
\item\label{VIass:item:iii}
We have $U := L^s(\Omega)$ with a fixed exponent
$s \in (1,\infty)$ satisfying $ s \geq q'$.
Here, $q'$ denotes the conjugate exponent
satisfying $1/q + 1/q' = 1$ (with $1/\infty := 0$).
The space $U$ is identified with a subset of $V^*$ via the (injective) embeddings 
$U = L^s(\Omega) \hookrightarrow L^q(\Omega)^* \hookrightarrow V^*$.
\item\label{VIass:item:iv}
 $A\colon V \to V^*$ is a linear and continuous operator which satisfies 
\begin{equation}
\label{eq:ellipticity}
\dual*{ Av}{ v}_V \geq c \|v\|_V^2\qquad \forall v \in V
\end{equation}
for some constant $c > 0$ and 
\begin{equation}
\label{eq:maxprop}
\min\parens*{
\dual*{ Av}{ \bracks{ v }_{a_1}^{a_2}}_V, \dual*{ A\bracks{ v }_{a_1}^{a_2}}{v}_V 
}
\geq
\dual*{A\bracks{v}_{a_1}^{a_2}}{ \bracks{v }_{a_1}^{a_2}}_V
\end{equation}
for all $v \in V$ and all $a_1, a_2 \in [-\infty, \infty]$
with $a_1 \leq 0 \leq a_2$.
\item\label{VIass:item:v} $f\colon \R \to \R$ is a nondecreasing, globally Lipschitz continuous, concave function.
 We identify $f$ with its
induced Nemytskii operator $f \colon V \to V^*$, i.e.,
\[
\langle f(v), w\rangle_V := \innerp*{ f(v)}{ w }_{L^2(\Omega)}\qquad \forall v,w \in V.
\]
\item\label{VIass:item:vi} 
 $K \subset V$ is a nonempty, closed, convex set satisfying 
\begin{equation}
\label{eq:Kconds}
\begin{aligned}
v  \in K, z \in V \quad &\Rightarrow \quad v + \max(0, z) \in K,
\\
v_1 , v_2 \in K \quad &\Rightarrow \quad \min(v_1, v_2)  \in K.
\end{aligned}
\end{equation}
\item $u \in V^*$ is a given parameter (the argument of the solution map).
\end{enumerate}
\end{assumption}\pagebreak

Note that, due to the assumption $q \geq 2$, the global Lipschitz continuity of $f$, the embedding $V \hookrightarrow L^q(\Omega)$, 
and the fact that $(\Omega, \Sigma, \mu)$ is finite, 
we have 
\begin{equation*}
\begin{aligned}
\abs*{\langle f(v), w\rangle_V }
&=
\abs*{\int_\Omega \parens*{ f(v) - f(0) + f(0)}w \, \d\mu }
\\
&\leq
\parens*{ C_1 \| v \|_{L^2(\Omega)} + |f(0)|\mu(\Omega)^{1/2}  }\| w \|_{L^2(\Omega)}
\leq
C_2\parens*{ \| v \|_{V} + 1 } \| w \|_{V}
\end{aligned}
\end{equation*}
for all $v,w \in V$ with some constants $C_1, C_2 \in \R$. The dual pairing in point \ref{VIass:item:v} 
is thus sensible. We begin by checking
that the solution map $S\colon u \mapsto y$ of \eqref{eq:V}
fits into the setting of \cref{ass:functionalanalytic}.

\begin{proposition}[Solvability]%
\label{prop:solvability}%
For all $u \in V^*$,
the variational inequality \eqref{eq:V} possesses a unique solution $S(u) := y \in V$.
The solution map $S\colon V^*\to V$, $u \mapsto y$, 
of \eqref{eq:V} is globally Lipschitz continuous, i.e., there exists a constant $C>0$ such that 
\begin{equation}
\label{eq:SLipschitz}
\norm*{
S(u_1) - S(u_2)
}_V
\leq
C \|u_1 - u_2\|_{V^*}\qquad \forall u_1, u_2 \in V^*.
\end{equation}
\end{proposition}

\begin{proof}
This follows immediately from \cite[Theorem~4-3.1]{Rodrigues1987}.
\end{proof}

Let us now define  $X := V^*$ and $Y := V$ and let $U = L^s(\Omega)$, $s$, and $q$
be as in \cref{ass:VIs}.
Then
it follows from our assumptions on $V$ and $s$ that $X$ is a separable Banach space, 
that $Y$ is a separable and reflexive Banach space that is continuously and compactly embedded 
into $L^q(\Omega)$, and that $U$ is a reflexive Banach space, cf.\ \cite[Theorems~5.2.11, 5.2.15]{Megginson1998}. 
From Schauder's theorem, the compactness, continuity, and density of the embedding 
$V \hookrightarrow L^q(\Omega)$,   
the finiteness of $(\Omega, \Sigma, \mu)$, and again our assumptions on $s$, we further obtain that 
$U$ embeds continuously and compactly into $X = V^*$.
In summary, this shows that the measure space $(\Omega, \Sigma, \mu)$ and the spaces $X = V^*$, $Y = V$, and $U$
associated with \eqref{eq:V}
satisfy
\cref{ass:functionalanalytic}\ref{ass:functionalanalytic:i}--\ref{ass:functionalanalytic:iv}.
Note that, from \eqref{eq:SLipschitz}, we also obtain that the solution operator 
$S\colon V^* = X \to Y = V$ of \eqref{eq:V} satisfies a local Lipschitz estimate of the form \eqref{eq:randomeq2625}. 
To see that $S$  fulfills the remaining conditions in \cref{ass:functionalanalytic}\ref{ass:functionalanalytic:v} too,
we note the following.

\begin{lemma}[Pointwise-a.e.\ convexity]%
\label{lemma:ptwconvex}%
The solution operator $S\colon V^* \to V$, $u \mapsto y$, of \eqref{eq:V}
is pointwise-a.e.\ convex, i.e., for all $u_1, u_2 \in V^*$ and all $\lambda \in [0,1]$, it holds 
\begin{equation*}
S(\lambda u_1 + (1 - \lambda)u_2) \leq \lambda S(u_1) + (1-\lambda)S(u_2)\quad\mu\text{-a.e.\ in }\Omega.
\end{equation*}
\end{lemma}
\begin{proof}
The proof follows standard lines, see, e.g., \cite[Lemma~6.3(iii)]{ChristofWachsmuth2021}. 
We include it for the convenience of the reader and since the setting in \cref{ass:VIs}
is slightly more general than what is typically considered in the literature. 
Suppose that $u_1, u_2 \in V^*$ and $\lambda \in [0,1]$ are given 
and set $y_1 := S(u_1)$, $y_2 := S(u_2)$, $y_{12} := S(\lambda u_1 + (1-\lambda)u_2)$,
and $w :=  y_{12} - \lambda y_1 - (1 - \lambda)y_2$. To prove the lemma,
we have to show that $w \leq 0$ holds $\mu$-a.e.\ in $\Omega$ or, equivalently, that $\max(0, w) = 0$  $\mu$-a.e.\ in $\Omega$.
To this end, we note that our assumptions on $V$ and $K$
imply that $y_1 + \max(0, w) \in K$ and $y_2 + \max(0, w) \in K$ holds and that 
\[
y_{12} - \max(0, w) = y_{12} - \max(0, y_{12} - \lambda y_1 - (1 - \lambda)y_2) = \min(y_{12}, \lambda y_1 + (1 - \lambda)y_2) \in K.
\]
The above allows us to use $y_1 + \max(0, w)$, $y_2 + \max(0, w)$, and $y_{12} - \max(0, w)$
as test functions in the VIs for $y_1$, $y_2$, and $y_{12}$, respectively. This yields
\begin{subequations}
\begin{align}
\dual*{ A y_1 + f(y_1)-  u_1 }{  \max(0, w) }_V \geq 0, \label{eq:trip:a)}
\\
\dual*{ A y_2 + f(y_2)-  u_2 }{  \max(0, w) }_V \geq 0, \label{eq:trip:b)}
\\
\dual*{ Ay_{12}  + f(y_{12})- \lambda u_1   - (1 - \lambda) u_2}{  -\max(0, w)  }_V \geq 0. \label{eq:trip:c)}
\end{align}
\end{subequations}
Note that, due to the definition of $w$ and the concavity and monotonicity of $f$,
we know that
\begin{equation*}
\begin{aligned}
&\dual{\lambda f(y_1) + (1 - \lambda)f(y_2) - f(y_{12}) }{\max(0, w) }_V
\\
&\qquad=
\int_\Omega \left (\lambda f(y_1) + (1 - \lambda)f(y_2) - f(y_{12}) \right )\max(0, w) \d\mu
\\
&\qquad\leq 
\int_\Omega \left (  f(\lambda y_1 + (1 - \lambda) y_2) - f(y_{12})  \right )\max(0, y_{12} - \lambda y_1 - (1 - \lambda)y_2) \d\mu
\\
&\qquad \leq 0,
\end{aligned}
\end{equation*} 
where the last inequality follows from a simple distinction of cases. This means 
that, by multiplying \eqref{eq:trip:a)} with $\lambda$ and \eqref{eq:trip:b)} with $(1 - \lambda)$,
by adding the resulting estimates to \eqref{eq:trip:c)}, and by subsequently exploiting \eqref{eq:ellipticity} and \eqref{eq:maxprop},
we obtain that 
\begin{equation*}
\begin{aligned}
	0 &\leq \dual{ Aw + f(y_{12})- \lambda f(y_1) - (1 - \lambda)f(y_2)}{  - \max(0, w) }_V
	\\
	&\leq - \dual*{ A w}{  \max(0, w) }_V
	\leq - \dual*{ A \max(0, w) }{  \max(0, w) }_V
	\leq - c \|\max(0, w) \|_V^2.
\end{aligned}
\end{equation*}
Thus, $\max(0, w) = 0$ $\mu$-a.e.\ and the proof is complete.
\end{proof}

In combination with our previous observations, 
\cref{lemma:ptwconvex} shows that the solution mapping $S\colon u \mapsto y$
of \eqref{eq:V} satisfies all of the conditions in \cref{ass:functionalanalytic} 
with $r = q$, see the comments before  \cref{def:Gateaux}.
To see that we can also consider exponents $r$ greater than $q$ in the situation of 
\eqref{eq:V} (and thus obtain Newton differentiability in stronger $L^p(\Omega)$-spaces by 
\cref{th:mainabstractsemismooth}), we
employ a generalized version of 
a well-known truncation argument of Stampacchia, see
\cite[Théorème~1]{Stampacchia1960}, \cite[Lemma~II.B2]{KinderlehrerStampacchia1980}. 
For the sake of reusability, we state this result in a format that makes it completely 
independent of \cref{ass:VIs}. 

\begin{lemma}\label{lemma:truncation}
	Suppose that $(\Omega, \Sigma, \mu)$ is a finite measure space with associated real Lebesgue spaces
	$(L^p(\Omega), \|\cdot\|_{L^p(\Omega)})$, $1 \leq p \leq \infty$.
	Let $q \in (1,\infty)$,  $s \in (1,\infty]$
	be exponents satisfying $\smash{\frac1s + \frac1q < 1}$ and $\smash{\frac1s  + \frac2q - 1 \ne 0}$, and 
	assume that 
	$u \in L^s(\Omega)$, $v \in L^q(\Omega)$
	are given
	such that
	the shrinkages $v_k := v -  \min\parens*{ k, \max \parens*{-k, v} } $, $k\geq 0$, satisfy 
	\begin{equation}
		\label{eq:randomeq278e46478}
		\norm{v_k}_{L^q(\Omega)}^2
		\leq 
		\alpha \int_\Omega \abs{ u v_k } \, \d\mu
		<
		\infty
		\qquad \forall k \geq k_0
	\end{equation}
	for some $k_0, \alpha \geq 0$. Define 
	$\sigma :=  \smash{\parens{\frac1s + \frac2q - 1}^{-1}}$.
	Then the following is true:
	\begin{enumerate}
		\item
			\label{lemma:truncation:1}
			In the case $\sigma < 0$,
			there exists a constant $C = C(s,q, \mu(\Omega)) > 0$ satisfying
			\begin{equation}
				\label{eq:Linfty2}
				\norm{v}_{L^\infty(\Omega)} \leq k_0 + C \alpha \norm{u}_{L^s(\Omega)}.
			\end{equation}
		\item
			\label{lemma:truncation:2}
			In the case $\sigma > 0$,
			there exists a constant $C = C(s,q,\mu(\Omega)) > 0$
			satisfying 
			\begin{equation}
				\label{eq:weak_L_sigma}
				\mu(\set{\omega \in \Omega \given \abs{v(\omega)} \ge k})
				\le
				C
				\frac{
					\alpha^\sigma \norm{u}_{L^s(\Omega)}^{\sigma}
					+
					k_0^{\sigma}
				}
				{k^\sigma}
				\qquad
				\forall k > k_0
			\end{equation}
			and 
			\begin{equation}
				\label{eq:Lr}
				\norm{v}_{L^r(\Omega)} \leq C
				\parens*{\frac{r}{\sigma - r}}^{1/r} \parens{ k_0 + \alpha \norm{u}_{L^s(\Omega)} }\qquad \forall r \in [1,\sigma).
			\end{equation}
	\end{enumerate}
\end{lemma}

Note that, in the case $\smash{\frac1s + \frac2q - 1 = 0}$, one can simply decrease $s$ slightly 
and then invoke point \ref{lemma:truncation:2} above. This then yields $v \in L^r(\Omega)$ for all $r \in [1,\infty)$.

\begin{proof}
	By rescaling $u$, it is enough to consider the case $\alpha = 1$.
We
define (up to sets of measure zero)
$L(k) := \set{\omega \in \Omega \given |v(\omega)| \geq k}$ for all $k\geq 0$. 
From
the definition of  $v_k$, it follows 
\begin{equation}
\label{eq:random2635}
	\norm{v_k}_{L^{q}(\Omega)}
	\geq 
	\parens[\bigg]{\int_{L(m)} (\abs{v} - k)^{q} \, \d\mu }^{1/q}
	\geq 
	(m - k) \mu(L(m))^{1/q}\qquad \forall m \geq k\geq 0,
\end{equation}
and, from Hölder's inequality
and
$\smash{\frac1q + \frac1s < 1}$,
we obtain 
\begin{equation}
\label{eq:random2635-2}
	\int_\Omega  \abs{u v_k}\, \d\mu 
	=
	\int_{L(k)} \abs{u v_k} \, \d\mu
	\leq \norm{u}_{L^s(\Omega)}  \mu(L(k))^{1 - 1/q -  1/s}  \norm{v_k}_{L^q(\Omega)}
	.
\end{equation}
In combination with \eqref{eq:randomeq278e46478}, 
the estimates \eqref{eq:random2635} and \eqref{eq:random2635-2} yield
\begin{equation*}
(m - k) \mu(L(m))^{1/q}
\leq
\norm{u}_{L^{s}(\Omega)}   \mu(L(k))^{1 - 1/q -  1/s}
\qquad\forall m \geq k\geq k_0.
\end{equation*}
This can be written as
\begin{equation*}
 \mu(L(m))  \leq 
 \norm{u}_{L^{s}(\Omega)}^q (m - k)^{-q} \mu(L(k))^\tau 
 \qquad\forall m > k\geq k_0
\end{equation*}
with 
\[
	\tau
	:= q \, \parens[\Big]{1 -  \frac{1}{q} - \frac{1}{s}}
	=
	q \, \parens[\Big]{ 1 - \frac2q - \frac1s } + 1
	=
	-\frac{q}{\sigma} + 1
	.
\]
We now distinguish between the cases \ref{lemma:truncation:1} and \ref{lemma:truncation:2}.
In case \ref{lemma:truncation:1},
we have $\tau > 1$ and may invoke 
\cite[Lemme~Préliminaire]{Stampacchia1960},
see also
\cite[Lemma~II.B1]{KinderlehrerStampacchia1980},
to deduce that
\begin{equation}
	\label{eq:Nullstelle}
	\mu(L(m)) = 0
	\quad\text{holds for}\quad
	m = 
	k
	+
	2^{\tau/(\tau-1)} \mu(L(k))^{(\tau-1)/q} \norm{u}_{L^{s}(\Omega)}
\end{equation}
whenever $k \ge k_0$.
Choosing $k = k_0$ in \eqref{eq:Nullstelle} yields \eqref{eq:Linfty2}. 
It remains to prove \ref{lemma:truncation:2}.
For this case, we have $\tau \in (0,1)$
and it follows from 
\cite[Lemme~Préliminaire]{Stampacchia1960}
that 
\begin{equation}
	\label{eq:weak_estimate}
	\mu(L(k))
	\le
	\hat C k^{-\sigma}
	\parens*{
		\norm{u}_{L^s(\Omega)}^{\sigma}
		+
		k_0^{\sigma} \mu(L(k_0))
	}
	\qquad\forall k > k_0
\end{equation}
holds with some constant $\hat C = \hat C(q,s) \ge 0$.
Note that the ``$+$'' on the right-hand side of \eqref{eq:weak_estimate} is (erroneously) missing in the statement of
\cite[Lemme~Préliminaire]{Stampacchia1960}.
By definition of $L(k)$, this yields \eqref{eq:weak_L_sigma}.
To prove \eqref{eq:Lr}, we suppose that $r \in [1,\sigma)$ is given and define 
$T^\sigma :=
	\hat C \,
	\parens[\big]{
		\norm{u}_{L^s(\Omega)}^{\sigma}
		+
		k_0^{\sigma} \mu(\Omega)
	}
$
and $k_1 := \max\set{ k_0, T}$.
Due to
\(
	r - \sigma - 1
	<
	-1
\)
and
$r - 1 \ge 0$,
we may employ a layer cake representation and \eqref{eq:weak_estimate} to get
\begin{align*}
	\norm{v}_{L^r(\Omega)}^r
	&=
	\int_0^\infty r \mu(L(k)) k^{r-1} \, \d k
	=
	\parens[\bigg]{\int_0^{k_1} + \int_{k_1}^\infty} r \mu(L(k)) k^{r-1} \, \d k
	\\
	&\le
	r \mu(\Omega)
	\int_0^{k_1}
	k^{r - 1}
	\, \d k
	+
	r T^\sigma
	\int_{k_1}^\infty
	k^{r - \sigma - 1}
	\, \d k
	\\
	&=
	 \mu(\Omega)  k_1^{r}
	-
	\frac{ r T^\sigma }{
		r - \sigma
	}
	k_1^{
		r - \sigma
	}
	\le
	\parens[\bigg]{
		 \mu(\Omega) 
		+
		\frac{r}{
			\sigma - r
		}
	}
	k_1^r
	.
\end{align*}
Plugging in the definition of $k_1$ and using trivial estimates now yields \eqref{eq:Lr}.
\end{proof}

\begin{remark}
	\label{rem:weak_lp}
	Note that the estimate \eqref{eq:weak_estimate} implies that $v$
	belongs to the weak Lebesgue space $L^{\sigma, \infty}(\Omega)$, cf.\ \cite{Dilworth2001}.
	The remaining part of the proof of \eqref{eq:Lr} above is a standard interpolation argument
	that ensures $v \in L^r(\Omega)$.
\end{remark}

\begin{remark}
	\label{rem:non_finite_measure}
	An estimate similar to inequality \eqref{eq:Linfty2} can also be obtained 
	for an infinite $\mu$. 
	However, for such a measure, the missing finiteness has to be compensated with some regularity
	of $v$, namely,
	$v \in L^p(\Omega)$ for some $p \in [1,\infty)$. Indeed, under this $L^p$-assumption,
	we obtain in the situation of \cref{lemma:truncation}
	from Chebyshev's inequality that the number $k_1 := \max\set{k_0, \norm{v}_{L^p(\Omega)}}$
	satisfies 
	\(
		\mu(L(k_1))
		\le
		k_1^{-p} \norm{v}_{L^p(\Omega)}^p
		\le
		1
	\).	
	Using this estimate in equation \eqref{eq:Nullstelle} with $k =k_1$ yields 
	that $\mu(L(m)) = 0$ holds for 
	\(
		m = k_1
		+
		2^{\tau/(\tau-1)} \mu(L(k_1))^{(\tau-1)/q} \norm{u}_{L^{s}(\Omega)}
		\le
		\max\set{k_0, \norm{v}_{L^p(\Omega)}}
		+
		2^{\tau/(\tau-1)} \norm{u}_{L^{s}(\Omega)}
	\).
	By the definition of $L(m)$, this gives
	\(
		\norm{v}_{L^\infty(\Omega)} \leq \max\set{k_0, \norm{v}_{L^p(\Omega)}} + C(s,q) \norm{u}_{L^s(\Omega)}.
	\)
\end{remark}

As a straightforward consequence of \cref{lemma:truncation}, we obtain the next result.

\begin{lemma}[Improved Lipschitz estimate]%
\label{lemma:improvLip}%
Let $q \in [2, \infty]$ and $s \in (1, \infty)$
be the exponents from \cref{ass:VIs}. Define 
$\kappa :=  \smash{ \frac1s + \frac2q - 1}$ and 
\begin{equation}
\label{eq:Rset}
\RR
:=
\begin{cases}
[1,\infty] & \text{ if } q \neq \infty, \frac1s + \frac1q < 1, \text{ and } \kappa < 0,
\\
[1,\infty) & \text{ if } q \neq \infty, \frac1s + \frac1q < 1, \text{ and } \kappa = 0,
\\
\left [1,\frac1\kappa\right ) & \text{ if } q \neq \infty, \frac1s + \frac1q < 1, \text{ and } \kappa > 0,
\\
[1, q] & \text{ else}. 
\end{cases}
\end{equation}
Then, for every $r \in \RR$, there exists a constant $C >0$ satisfying
\begin{equation}
\label{eq:improvedLipschitz}
\norm*{ S(u + z) - S(u)
}_{L^r(\Omega)} \leq C \|z\|_{L^s(\Omega)}
\qquad\forall u \in V^*, z \in U.
\end{equation}
\end{lemma}
\begin{proof}
The ``else''-case follows from \eqref{eq:SLipschitz}, the finiteness of $\mu$, and the embeddings
$V \hookrightarrow L^q(\Omega)$ and $U \hookrightarrow V^*$.
To prove \eqref{eq:improvedLipschitz} 
in the remaining cases, we suppose that $u \in V^*$ and $z \in U$ are given, 
define $y_1 := S(u)$ and $y_2 := S(u + z)$,
and set
\[
(y_1 - y_2)_k  := y_1 - y_2 - [y_1 - y_2]_{-k}^k
\qquad\forall k \ge 0.
\]
From 
\[
y_1 - (y_1 - y_2)_k 
=
y_2 + \bracks{y_1 - y_2}_{-k}^k
=
\begin{cases}
y_2 + k & \text{ if } y_1 \geq k + y_2,
\\
y_2 - k & \text{ if } y_1  \leq y_2 - k,
\\
y_1 & \text{ if } |y_1 - y_2| < k,
\end{cases}
\]
it follows that $y_1 - (y_1 - y_2)_k  \geq \min(y_1, y_2)$ holds $\mu$-a.e.\ in $\Omega$.
Due to our assumptions on $K$, this implies $y_1 - (y_1 - y_2)_k \in K$.
Analogously, we also get $y_2 + (y_1 - y_2)_k \in K$.
By using these functions 
as test functions in the VIs for $y_1$ and $y_2$, respectively,
and by exploiting  the monotonicity of $f$, we obtain
\begin{equation*}
\begin{aligned}
0
&\leq
\dual*{ Ay_1 + f(y_1) - u}{ - (y_1 - y_2)_k }_V + \dual*{ Ay_2+ f(y_2) - u - z}{ (y_1 - y_2)_k}_V
\\
&\leq 
\dual*{ A(y_1 - y_2)}{ - (y_1 - y_2)_k }_V + \dual*{ -z}{ (y_1 - y_2)_k}_V,
\end{aligned}
\end{equation*}
which can also be written as 
\begin{equation}
\label{eq:randomeq2763gd36}
\dual*{ A(y_1 - y_2)}{  (y_1 - y_2)_k }_V \leq - \int_\Omega z (y_1 - y_2)_k \, \d\mu.
\end{equation}
Since 
\begin{equation*}
\begin{aligned}
\dual*{ Av}{  v - [v]_{-k}^k }_V 
&=
\dual*{ A(v - [v]_{-k}^k)}{  v - [v]_{-k}^k }_V 
+
\dual*{ A [v]_{-k}^k}{  v }_V - \dual*{ A [v]_{-k}^k}{  [v]_{-k}^k}_V
\\
&\geq c \|v - [v]_{-k}^k \|_V^2 + 0
\qquad \forall v \in V, k \geq 0
\end{aligned}
\end{equation*}
holds for a constant $c >0$ by \cref{ass:VIs}\ref{VIass:item:iv}, \eqref{eq:randomeq2763gd36},
the embedding $V \hookrightarrow L^q(\Omega)$, 
and the definition of $(y_1 - y_2)_k$ imply that there exists a constant $C>0$ with
\begin{equation*}
\|(y_1 - y_2)_k \|_{L^q(\Omega)}^2 \leq C\int_\Omega | z (y_1 - y_2)_k| \, \d\mu\qquad \forall k \geq 0. 
\end{equation*}
To complete the proof, it now suffices to invoke \cref{lemma:truncation}.
\end{proof}

\begin{remark}
	\label{rem:marcinkiewicz}
	In his seminal work \cite{Stampacchia1960},
	Stampacchia used the celebrated Mar\-cin\-kie\-wicz
	interpolation theorem
	\cite{Zygmund1989}
	to get a similar result
	for linear equations
	(even including the critical exponent $1/\kappa$ in the third case of \eqref{eq:Rset}). 
	It is not clear whether
	this interpolation theorem
	applies to $S$ 	in the situation of \cref{ass:VIs}.
\end{remark}

With \cref{lemma:ptwconvex,lemma:improvLip} in place, we are in the position 
to state the consequences of the analysis in \cref{sec:2} for the solution map $S$
of \eqref{eq:V}. Note that, in the situation of 
\eqref{eq:V}, the strong-weak Bouligand differential of $S$ at a point $u\in V^*$ is the subset 
of $\LL(V^*, V)$ given by 
\begin{equation}
	\label{eq:BouligandDiffVI}
	\partial_B^{sw}S(u)
	=
	\set[\big]{
	G 
	\given \exists \{u_n\} \subset \DD_S \colon u_n \to u \text{ in }V^*,~ S'(u_n) \WOT G \text{ in }\LL(V^*, V)}.
\end{equation}

\begin{corollary}[Semismoothness of the solution map of \eqref{eq:V}]%
\label{cor:VIsemismooth1}%
Let $q \in [2, \infty]$ and $s \in (1, \infty)$
be the exponents from \cref{ass:VIs}. Define
$\kappa :=  \smash{ \frac1s + \frac2q - 1}$ and
\begin{equation*}
\PP
:=
\begin{cases}
[1,\infty) & \text{ if } q \neq \infty, \frac1s + \frac1q < 1, \text{ and } \kappa \leq 0,
\\
\left [1,\frac1\kappa\right) & \text{ if } q \neq \infty, \frac1s + \frac1q < 1, \text{ and } \kappa > 0,
\\
[1, q] & \text{ else}. 
\end{cases}
\end{equation*}
Then the solution map $S\colon V^* \to V$ of the variational inequality \eqref{eq:V}
with the strong-weak Bouligand differential 
$\partial_B^{sw}S\colon V^* \rightrightarrows\LL(V^*, V)$ in \eqref{eq:BouligandDiffVI}
is Newton differentiable 
w.r.t.\ perturbations in $U = L^s(\Omega)$
with values in $L^p(\Omega)$ for all $p \in \PP$.
\end{corollary}
\begin{proof}
Since $(\Omega, \Sigma, \mu)$ and the spaces $X := V^*$, $Y := V$, and $U := L^s(\Omega)$ satisfy all of the conditions 
in points \ref{ass:functionalanalytic:i}, \ref{ass:functionalanalytic:ii}, \ref{ass:functionalanalytic:iii}, and \ref{ass:functionalanalytic:iv} 
of \cref{ass:functionalanalytic} (cf.\ the comments after \cref{prop:solvability,lemma:ptwconvex})
and since $S\colon V^* = X \to Y = V$ satisfies 
the conditions in \cref{ass:functionalanalytic}\ref{ass:functionalanalytic:v} for all exponents $r$
in the set $\RR$ defined in \eqref{eq:Rset} by \cref{prop:solvability,lemma:ptwconvex,lemma:improvLip}, the
assertion of the corollary follows immediately from \cref{th:mainabstractsemismooth} 
and the finiteness of the measure space $(\Omega, \Sigma, \mu)$.
\end{proof}

\section{Tangible examples of variational inequalities  covered by our analysis}
\label{sec:4}
To make the results of \cref{sec:2,sec:3} more accessible,
we collect some examples of VIs that are 
covered by \cref{cor:VIsemismooth1}.

\subsection{The scalar Signorini problem}%
\label{subsec:Signorini}%
As a first example, we consider the scalar Signorini problem:
Assume that $\Omega \subset \R^d$, $2 \leq d \in \mathbb{N}$,
is a bounded Lipschitz domain that is endowed with the Lebesgue measure and whose boundary $\partial \Omega$
is decomposed disjointly into three (possibly empty) 
measurable
parts $\Gamma_D$, $\Gamma_N$, and 
$\Gamma_S$. Define 
\begin{equation*}
V :=
\set*{
v \in H^1(\Omega) \given \tr(v) = 0 \text{ a.e.\ on } \Gamma_D 
}\quad\text{and}\quad
\|\cdot \|_V := \|\cdot \|_{H^1(\Omega)},
 \end{equation*}
where $(H^1(\Omega), \|\cdot\|_{H^1(\Omega)})$ is defined as usual
and where $\tr\colon H^1(\Omega) \to L^2(\partial \Omega)$ denotes the trace operator, 
see \cite{Attouch2006,Necas2012}. 
Suppose further that a measurable function $\psi\colon \partial \Omega \to \R$ is given such that
\[
K := 
\set*{
v \in V \given \tr(v) \geq \psi \text{ a.e.\ on } \Gamma_S
}
\]
is nonempty. For right-hand sides $u \in V^*$, 
we consider the Signorini-type VI
\begin{equation}
\label{eq:Signorini}
y \in K,\qquad 
\innerp*{ y}{ v - y }_{H^1(\Omega)} -\dual*{ u}{ v - y}_{V} \geq 0~~\forall v \in K. 
\end{equation}
Note
that $(V, \|\cdot\|_{V})$ 
is a separable Hilbert space that embeds continuously, compactly, and densely into $L^q(\Omega)$ for all 
$2 \leq q < 2d/(d-2)$ due to the properties of $H^1(\Omega)$, see \cite[Theorem~6.1]{Necas2012}.
Here, the right-hand side of the  inequality $2 \leq q < 2d/(d-2)$ is understood as $\infty$ for $d=2$.
From \cite[Theorem~5.8.2]{Attouch2006}, we also obtain that 
\[
\bracks*{ v }_{a_1}^{a_2} = \min\parens*{ a_2, \max \parens*{a_1, v} } \in V
\]
and
\[
\parens*{
\bracks*{ v }_{a_1}^{a_2} , v
}_{H^1(\Omega)}
=
\parens*{
v, \bracks*{ v }_{a_1}^{a_2} 
}_{H^1(\Omega)}
\geq
\parens*{
\bracks*{ v }_{a_1}^{a_2}  , \bracks*{ v }_{a_1}^{a_2} 
}_{H^1(\Omega)}
\]
holds for all $v \in V$ and all  $a_1, a_2 \in [-\infty, \infty]$
with $a_1 \leq 0 \leq a_2$, and that $K$ satisfies \eqref{eq:Kconds}. Since the bilinear form in \eqref{eq:Signorini}
is trivially elliptic, this shows that \eqref{eq:Signorini} 
satisfies all of the conditions in \cref{ass:VIs} (with $f\equiv 0$)
provided $q$ and $s$ are chosen such that $2 \leq q < 2d/(d-2)$,
$1 < s < \infty$, and $s \geq  (1 - 1/q)^{-1}$ holds. 
In combination with the analysis of \cref{sec:3}, this allows us to obtain the following result.

\begin{corollary}[Semismoothness of the solution map of the Signorini problem]%
\label{cor:signorini}%
The problem \eqref{eq:Signorini}  possesses a well-defined solution operator
$S\colon V^* \to V$, $u \mapsto y$.
If $s \in (1, \infty)$ is a fixed exponent satisfying $s > 2d/(d+2)$ and if $\PP$
is defined by 
\begin{equation}
\label{eq:PH1def}
\PP
:=
\begin{cases}
[1,\infty) & \text{ if }  s \geq \frac{d}{2},
\\
\left [1, \parens*{\frac1s - \frac2d}^{-1}\right) & \text{ if }  s < \frac{d}{2},
\end{cases}
\end{equation}
then this solution operator $S\colon V^* \to V$
is Newton differentiable 
w.r.t.\ perturbations in $L^s(\Omega)$
with values in $L^p(\Omega)$ for all $p \in \PP$
when endowed with the 
strong-weak Bouligand differential 
$\partial_B^{sw}S\colon V^* \rightrightarrows\LL(V^*, V)$ defined in \eqref{eq:BouligandDiffVI}.
\end{corollary}

\begin{proof}
As all of the conditions in \cref{ass:VIs} are satisfied in the situation of \eqref{eq:Signorini}
(with $q := 2d/(d-2) - \varepsilon$, $\varepsilon > 0$ arbitrarily small),
the assertions of the corollary follow immediately from 
\cref{prop:solvability,cor:VIsemismooth1}. 
\end{proof}

\subsection{The classical obstacle problem}\label{subsec:obstacle}
As a second example, we consider the classical obstacle problem:
Suppose that $\Omega \subset \R^d$, $2 \leq d \in \mathbb{N}$,
is a bounded, nonempty, open set that is endowed with the Lebesgue measure. 
We assume that a measurable function $\psi\colon \Omega \to \R$
is given such that the set 
\[
K := \set*{
v \in H_0^1(\Omega)
\given
v \geq \psi
\text{ a.e.\ in } \Omega
}
\]
is nonempty. 
Here, $H_0^1(\Omega)$ is (as usual) defined to be the Hilbert space that is obtained by taking 
the closure of $C_c^\infty(\Omega)$ in $(H^1(\Omega), \|\cdot\|_{H^1(\Omega)})$, 
see \cite[section~5.1]{Attouch2006}. For given $u \in H^{-1}(\Omega) := H_0^1(\Omega)^*$, 
we are interested in the classical obstacle problem 
\begin{equation}
\label{eq:classobstacle}
 y \in K,\qquad 
\dual*{ - \Delta y - u}{ v - y }_{H_0^1(\Omega)} \geq 0~~\forall v \in K,
 \end{equation}
 where $\Delta \in \LL(H_0^1(\Omega), H^{-1}(\Omega))$ denotes the distributional Laplacian. 
Analogously to  \cref{subsec:Signorini}, we obtain
that the space $V := H_0^1(\Omega)$
associated with \eqref{eq:classobstacle} is separable and Hilbert 
and that $H_0^1(\Omega) \hookrightarrow L^q(\Omega)$ holds continuously, compactly, and densely for all
$2 \leq q < 2d/(d-2)$.
(Note that no regularity of $\Omega$ is needed for the embedding here
due to the zero boundary conditions.) From \cite[Theorems~5.3.1, 5.8.2]{Attouch2006}, 
it also again follows that the space $V = H_0^1(\Omega)$,
the operator $A := -\Delta$, and the set $K$
satisfy all of the remaining conditions in points \ref{VIass:item:ii}, \ref{VIass:item:iv}, and \ref{VIass:item:vi} of \cref{ass:VIs}. 
This shows that the standing assumptions of \cref{sec:3} are all satisfied 
by \eqref{eq:classobstacle} (with $f \equiv 0$ and for all 
$q$ and $s$ with $2 \leq q < 2d/(d-2)$, $1 < s < \infty$, and $s \geq  (1 - 1/q)^{-1}$).
Invoking \cref{cor:VIsemismooth1} 
now yields the following counterpart of \cref{cor:signorini}. 

\begin{corollary}[Semismoothness of the solution map of the obstacle problem]%
\label{cor:obstacle}%
The problem \eqref{eq:classobstacle}  possesses a well-defined solution operator
$S\colon H^{-1}(\Omega) \to H_0^1(\Omega)$, $u \mapsto y$.
If $s \in (1, \infty)$ is a fixed exponent satisfying $s > 2d/(d+2)$ and if $\PP$
is defined as in \eqref{eq:PH1def},
then this solution map $S\colon H^{-1}(\Omega) \to H_0^1(\Omega)$
with the 
strong-weak Bouligand differential 
$\partial_B^{sw}S\colon H^{-1}(\Omega) \rightrightarrows\LL(H^{-1}(\Omega) , H_0^1(\Omega))$
is Newton differentiable 
w.r.t.\ perturbations in $L^s(\Omega)$
with values in $L^p(\Omega)$ for all $p \in \PP$.
\end{corollary}

\begin{proof}
This follows immediately from 
\cref{prop:solvability,cor:VIsemismooth1} and the same arguments as in \cref{subsec:Signorini}.
\end{proof}

Note that, in the special case $s \in [3/2, \infty)$ and $d = 3$, \cref{cor:obstacle} yields that 
the solution operator $S\colon H^{-1}(\Omega) \to H_0^1(\Omega)$ of \eqref{eq:classobstacle} is 
Newton differentiable in the sense that, for all $u \in H^{-1}(\Omega)$ and  all $p \in [1, \infty)$, we have
\[
\sup_{G \in \partial_B^{sw}S(u + z)} \frac{\|S(u + z) - S(u) - Gz\|_{L^p(\Omega)}}{\|z\|_{L^s(\Omega)}} \to 0
\quad\text{ for } \|z\|_{L^s(\Omega)} \to 0.
\]
What is remarkable here is that this result holds for all $p \in [1, \infty)$
even in those cases where the
obstacle $\psi $ in \eqref{eq:classobstacle} satisfies
$0 \leq \psi \in H_0^1(\Omega) \setminus L^{6 + \varepsilon}(\Omega)$ for some $\varepsilon > 0$
and where, as a consequence, $K \cap L^{6 + \varepsilon}(\Omega) = \emptyset$
and $S(H^{-1}(\Omega)) \cap  L^{6 + \varepsilon}(\Omega) = \emptyset$ holds. 
Even if there are no states $S(u)$ satisfying $S(u) \in L^p(\Omega)$ for all $p \in [1, \infty)$,
the solution mapping $S\colon H^{-1}(\Omega) \to H_0^1(\Omega)$ of \eqref{eq:classobstacle} can thus still be 
Newton differentiable with values in $L^p(\Omega)$ for all $p \in [1, \infty)$. Capturing these effects is the main motivation 
for considering different regularities for $x$ and $z$ in \cref{def:NewtDif}.

We remark that,
for sufficiently regular obstacles $\psi$ and states $y$, 
the strong-weak Bouligand differential  of the solution map $S\colon H^{-1}(\Omega) \to H_0^1(\Omega)$ of  \eqref{eq:classobstacle}
has been characterized completely in \cite[Theorem~5.6]{Rauls2020}.
We recall this result for the convenience of the reader and since we will use it in \cref{sec:6}.

\begin{theorem}[{\cite[Theorem~5.6]{Rauls2020}}]%
\label{th:RaulsWachsmuth}%
Suppose that $\psi \in C(\overline{\Omega}) \cap H^1(\Omega)$ holds 
and that $\psi < 0$ on $\partial \Omega$.
Assume further that $u \in H^{-1}(\Omega)$ is given such that the 
solution of \eqref{eq:classobstacle} satisfies $y := S(u) \in C(\overline{\Omega})$.
Then the strong-weak Bouligand differential $\partial_B^{sw}S(u)$ of $S\colon H^{-1}(\Omega) \to H_0^1(\Omega)$
at $u$, i.e., the subset of $\LL(H^{-1}(\Omega), H_0^1(\Omega))$ defined by \eqref{eq:BouligandDiffVI} 
with $V:= H_0^1(\Omega)$, is given by 
\begin{equation}
\label{eq:DBSobstacle}
\partial_B^{sw}S(u) := 
\set*{
G_\nu\given
\nu \in \MM_0(\Omega),~ \nu(I(u)) = 0,~\nu = + \infty \text{ on } A_s(u)
}.
\end{equation}
Here, $\MM_0(\Omega)$ denotes the set of all capacitary measures on $\Omega$, see 
\cite[Definition~3.1]{Rauls2020}, $I(u) := \set{\omega \in \Omega \given y(\omega) > \psi(\omega)}$
denotes the inactive set of $u$, 
$A_s(u)$ denotes the strictly active set of $u$ as defined in \cite[section~2.2]{Rauls2020},
and $G_\nu \in \LL(H^{-1}(\Omega), H_0^1(\Omega)) $, $\nu \in \MM_0(\Omega)$,
denotes the solution map $H^{-1}(\Omega) \ni z \mapsto w \in H_0^1(\Omega)$ of the relaxed Dirichlet problem 
\[
w\in H_0^1(\Omega),\quad 
-\Delta w + \nu w = z,
\]
defined as in \cite[Equation~(10)]{Rauls2020}.
\end{theorem}

Together, 
\cref{cor:obstacle} and \cref{th:RaulsWachsmuth} 
provide a readily applicable framework for the development of numerical solution algorithms 
based on the semismoothness properties of the solution operator of the obstacle problem,
see \cref{sec:5,sec:6}. In particular, the description of $\partial_B^{sw}S(u)$ in \eqref{eq:DBSobstacle} is 
also amenable to classical adjoint-based approaches as used, for instance, in \cite[section~4]{Christof2018}. 
We would like to point out that the assumption 
$y := S(u) \in C(\overline{\Omega})$ in \cref{th:RaulsWachsmuth}
is not very restrictive 
as the continuity of the solutions of \eqref{eq:classobstacle} can often be ensured easily  by 
invoking $W^{2,p}$-regularity results, see \cite[section~IV-2]{KinderlehrerStampacchia1980}.
If, for example, $\Omega \subset \R^d$ is a bounded convex domain with $d \le 3$ and  
$\psi$ satisfies $\psi \in H^2(\Omega) \subset C(\overline{\Omega})$ and $\psi < 0$ on $\partial \Omega$,
then it follows from the approach in \cite[section~IV-2]{KinderlehrerStampacchia1980}
and  \cite[Theorem~3.2.1.2]{Grisvard1985}
that $S(u) \in H^2(\Omega) \subset C(\overline{\Omega})$ holds for all $u \in L^2(\Omega)$,
and we may deduce from \cref{cor:obstacle} that the solution map $S$ of \eqref{eq:classobstacle}
is Newton differentiable as a function $S\colon L^2(\Omega) \to L^p(\Omega)$ for all $1 \leq p < \infty$
in the sense that 
\[
\sup_{G \in \partial_B^{sw}S(u + z)} \frac{\|S(u + z) - S(u) - Gz\|_{L^p(\Omega)}}{\|z\|_{L^2(\Omega)}} \to 0
\quad\text{for } \|z\|_{L^2(\Omega)} \to 0
\]
holds for all $u \in L^2(\Omega)$ and all $1 \leq p < \infty$ with the differential $\partial_B^{sw}S(u)$ 
given by \cref{eq:DBSobstacle} for all $u \in L^2(\Omega)$. 
Note that, although the operator $S$ is considered purely on $L^2(\Omega)$ here, 
the generalized differential in the semismoothness result is still the whole strong-weak Bouligand 
differential in $\LL(H^{-1}(\Omega), H_0^1(\Omega))$ as defined in \eqref{eq:BouligandDiffVI}.
This shows that, although the control space is typically chosen as a Lebesgue space in applications, 
it is very natural to study generalized differentials of solution operators 
of obstacle-type VIs in the dual of the underlying Hilbert space. 

\subsection{Comments on further examples}
Before we demonstrate that the results of \cref{sec:2,sec:3} 
can indeed be used to design solution algorithms for optimal control problems, 
we would like to emphasize that \cref{th:mainabstractsemismooth,cor:VIsemismooth1} 
are not only applicable to the Signorini problem \eqref{eq:Signorini} and the obstacle problem \eqref{eq:classobstacle},
but also to various other VIs. 
It is, for instance, straightforward to check that the thin obstacle problem
as discussed in \cite[section~8:7]{Rodrigues1987} and 
obstacle-type VIs formulated in $H_0^s(\Omega)$, $0 < s < 1$,
are covered by our analysis, cf.\ \cite[Exemple~3]{Mignot1976} and \cite[Corollary~3.3]{ChristofWachsmuth2018}.
Since we may also choose $K = V$ in \cref{ass:VIs}, \cref{cor:VIsemismooth1} also immediately yields 
semismoothness results for certain semilinear PDEs. (For those, however, the Newton differentiability of the solution map
can also be established easily in a direct manner.) We omit discussing these examples in more detail  here.

\section{An application in optimal control}
\label{sec:5}
In this section, we are concerned with the following setting.

\begin{assumption}[Standing assumptions for the analysis of \cref{sec:5}]%
\label{ass:optctrl}%
Throughout this section, we assume the following:
\begin{enumerate}
\item $(\Omega, \Sigma, \mu)$ is as in \cref{ass:VIs}\ref{VIass:item:i}.
\item\label{item:ass:optctrl:ii}
 $(V, \|\cdot\|_V)$ is a real Hilbert space 
that satisfies the conditions in \cref{ass:VIs}\ref{VIass:item:ii} with $q = 2$.
We interpret the spaces $V$, $L^2(\Omega)$, and $V^*$
as a Gelfand triple, i.e., $V \hookrightarrow  L^2(\Omega) \cong L^2(\Omega)^* \hookrightarrow V^*$. 
\item $A \in \LL(V, V^*)$ satisfies 
\cref{ass:VIs}\ref{VIass:item:iv} and is symmetric, i.e., 
\begin{equation*}
\dual*{ Av}{ w }_V = \dual*{ Aw}{ v }_V\qquad \forall v,w \in V.
\end{equation*}
\item  $(W, \|\cdot\|_W)$ is a real Hilbert space that satisfies $W \hookrightarrow L^2(\Omega)$ continuously and densely. 
We interpret the spaces $W$, $L^2(\Omega)$, and $W^*$
as a Gelfand triple, i.e., $W \hookrightarrow  L^2(\Omega) \cong L^2(\Omega)^* \hookrightarrow W^*$.
\item $L\colon W \to W^*$ is a linear and continuous operator with inverse $P := L^{-1}$.
\item 
 $K \subset V$ is a set that satisfies the conditions in  \cref{ass:VIs}\ref{VIass:item:vi}.
 \item $y_D \in L^2(\Omega)$ is a given desired state and $\alpha > 0$ is a given Tikhonov parameter.
\end{enumerate}
\end{assumption}

In the above situation, we consider the optimization problem
\begin{equation}
\label{eq:O}
\tag{\textup{OC}}
	\left \{\,\,
	\begin{aligned}
		\text{Minimize} \quad & J(y, u) := \frac12 \norm*{ y - y_D}_{L^2(\Omega)}^2 + \frac{\alpha}{2} \dual*{ A u}{ u }_V \\
		\text{w.r.t.}\quad &u \in V,\quad y \in W,\\
		\text{s.t.}\quad & Ly = u \text{ in } W^*\\
		\text{and}\quad&  u \in K.
	\end{aligned}
	\right.
\end{equation}
Note that this problem can be interpreted as an abstract optimal control problem with unilateral 
control constraints posed in the space $V$, see the tangible example in \cref{sec:6}. 
The next result is standard.

\begin{proposition}[Unique solvability of \eqref{eq:O}]%
The optimization problem \eqref{eq:O} possesses a unique solution $\bar u \in V$ with associated 
state $\bar y := P \bar u \in W$. This solution is uniquely characterized by the following stationarity system:
\begin{equation}
\label{eq:Ooptsys}
\begin{aligned}
\bar y, \bar z &\in W, \qquad \bar z = P^*(\bar y - y_D),\qquad \bar y = P \bar u,
\\
\bar u &\in K, \qquad \dual*{ A\bar u + \alpha^{-1} \bar z}{ v - \bar u }_V\geq 0~~\forall v\in K.
\end{aligned}
\end{equation}
\end{proposition}

Here and in what follows, $P^* \in \LL(W^*, W)$ 
is the adjoint of $P$, i.e.,
\[
\dual*{ w_1^*}{ Pw_2^* }_W = \dual*{ w_2^*}{ P^* w_1^* }_W \qquad \forall w_1^*, w_2^* \in W^*.
\]

\begin{proof}
The unique solvability of \eqref{eq:O} follows from the direct method of 
the calculus of variations and the strict convexity of $J$.
That $\bar u$ is uniquely characterized by \eqref{eq:Ooptsys} is a consequence of 
standard calculus rules for the convex subdifferential. 
\end{proof}

Note that the VI in \eqref{eq:Ooptsys} has precisely the form 
\eqref{eq:V} with $f \equiv 0$ and right-hand side $-\bar z / \alpha$. 
Henceforth, we will denote the solution operator of this 
inequality, i.e., the function that maps a right-hand side $z \in V^*$
(or $z \in L^2(\Omega) \hookrightarrow V^*$ or $z \in W \hookrightarrow L^2(\Omega) \hookrightarrow V^*$, respectively)
to the solution $w \in V$ of the problem 
\begin{equation}
\label{eq:VI22}
w \in K,\qquad 
\dual*{ Aw - z}{ v - w }_V \geq 0\qquad \forall v \in K,
\end{equation}
with $S$. With this notation, the system \eqref{eq:Ooptsys} can be recast as 
\begin{equation}
\label{eq:optimalityconditionshort}
\bar u \in V,\qquad  \bar y, \bar z \in W,\qquad \bar z = P^*(\bar y - y_D),\qquad \bar y = P \bar u,
\qquad \bar u = S\parens*{- \alpha^{-1}\bar z},
\end{equation}
or, equivalently, after eliminating all variables except $\bar y$, as 
\begin{equation}
\label{eq:optimalityconditionveryshort}
\bar y - P S\parens*{\alpha^{-1} P^*( y_D - \bar y)} = 0.
\end{equation}
This reformulation of the necessary and sufficient optimality condition \eqref{eq:Ooptsys}  
can be used as a point of departure for setting up a semismooth Newton method
for the numerical solution of \eqref{eq:O} based on \cref{cor:VIsemismooth1}.

\begin{Algorithm}[Semismooth Newton method for the solution of \eqref{eq:O}]\label{alg:semiNewton}
~\hspace{-10cm}
\begin{algorithmic}[1]
  \STATE Choose an initial guess $y_0 \in L^2(\Omega)$ and a tolerance $\texttt{tol} \geq 0$.
    \FOR{$i = 0,1,2,3,...$}
    \STATE Calculate $\zeta_i :=P^*( y_D - y_i)/\alpha$, $u_i := S(\zeta_i)$, and $\tilde y_i := Pu_i$.
        \IF{$\|y_i - \tilde y_i\|_{L^2(\Omega)} \leq \texttt{tol}$} \label{algo1:step:4}
            \STATE STOP the iteration (convergence reached). 
        \ELSE
        \STATE 
            Choose an element $G_i$ of the differential $\partial_B^{sw}S(\zeta_i)$ defined in \eqref{eq:BouligandDiffVI}.\label{algo1:step:7}
        \STATE Determine $y_{i+1} \in L^2(\Omega)$ by solving the linear equation\label{algo1:step:8}
        \begin{equation*}
        	y_{i + 1} + \alpha^{-1} P G_i P^* y_{i+1} = \tilde y_i + \alpha^{-1} P G_i P^* y_i.
        \end{equation*}
        \ENDIF
    \ENDFOR
\end{algorithmic}
\end{Algorithm}

To see that \cref{alg:semiNewton} is sensible, we note the following.

\begin{lemma}
\label{lemma:Gcoercivity}
Suppose that $u \in V^*$ and $G \in \partial_B^{sw}S(u)$ are given. Then it holds 
\[
\dual*{  z}{ Gz}_V \geq 0\qquad \forall z \in V^*. 
\]
\end{lemma}
\begin{proof}
We first assume that $u \in V^*$ is a point of Gâteaux differentiability of $S\colon V^* \to V$. 
From the definition of $S$ via \eqref{eq:VI22}, we get
\begin{equation*}
\begin{aligned}
	\dual{A S(u + t z) - (u + t z)}{S(u) - S(u + t z)}_V &\ge 0 \\
	\text{and} \qquad
	\dual{A S(u      ) -  u       }{S(u + t z) - S(u)}_V &\ge 0
\end{aligned}
\end{equation*}
for all $z \in V^*$ and $t>0$.
Adding these inequalities leads to
\[
	\dual{t z       }{S(u + t z) - S(u)}_V \ge
	\dual{A (S(u+t z)- S(u))}{S(u + t z) - S(u)}_V \ge 0
	.
\]
Now, we can divide by $t^2$ and pass to the limit $t \to 0^+$
to arrive at
the claim of the lemma for the special case that $G = S'(u)$ is a Gâteaux derivative. 

In the general case
let $u \in V^*$, $G \in \partial_B^{sw}S(u)$, and $z \in V^*$ be given.
Suppose that $\{u_n\} \subset V^*$
is an approximating sequence of Gâteaux points for $G$ as in \eqref{eq:BouligandDiffVI}.
Then $S'(u_n)z \weakly Gz$ in $V$ as $n \to \infty$
and the inequality $\dual*{ z }{ S'(u_n)z}_V \geq 0$ for all $n$ yield
\[
0 \leq 
\dual*{ z }{ S'(u_n)z}_V
\to
\dual*{  z}{ Gz}_V.
\]
\end{proof}

Using \cref{lemma:Gcoercivity}, we can prove that the linear equation 
that has to be solved in \cref{algo1:step:8} of \cref{alg:semiNewton} always possesses
a unique solution. 

\begin{proposition}[Feasibility of the semismooth Newton step]%
	\label{prop:unif_bdd_inv}%
	For every $\zeta \in V^*$
	and
	$G \in \partial_B^{sw}S(\zeta)$,
	the operator
\[
	\operatorname{Id} + \alpha^{-1} PGP^* \colon L^2(\Omega) \to L^2(\Omega)
\]
is an isomorphism and the norm of its inverse is bounded by $1$.
\end{proposition}
\begin{proof}
	This follows from \cref{lemma:Gcoercivity} and the lemma of Lax--Milgram.
\end{proof}
 
The local convergence of  \cref{alg:semiNewton}
now follows from standard arguments.

\begin{theorem}[Local superlinear convergence of \cref{alg:semiNewton}]\label{th:NewtonConvergence}%
Let $\bar u \in V$ be the optimal control of \eqref{eq:O} and $\bar y = P\bar u \in W$ the associated optimal state.
There exists $\varepsilon> 0$ such that, for every $y_0 \in L^2(\Omega)$ with 
$\|y_0 - \bar y\|_{L^2(\Omega)} < \varepsilon$,
\cref{alg:semiNewton} with $\texttt{tol} = 0$ either terminates after finitely many steps with the solution of \eqref{eq:O} or 
produces sequences $\{y_i\} \subset L^2(\Omega)$, $\{u_i\} \subset V$, and 
$\{\tilde y_i\} \subset W$ that satisfy 
\begin{equation*}
\begin{aligned}
y_i &\to \bar y \text{ q-superlinearly in } L^2(\Omega),
\\
u_i &\to \bar u  \text{ r-superlinearly in } V, \text{ and }
\\
\tilde y_i &\to \bar y \text{ r-superlinearly in } W. 
\end{aligned}
\end{equation*}
\end{theorem}
\begin{proof}
	The operator $L^2(\Omega) \ni z \mapsto z - P S( P^*(y_D - z)/\alpha) \in L^2(\Omega)$
	on the left-hand side of \eqref{eq:optimalityconditionveryshort}
	is semismooth
	since $S$ with $\partial_B^{sw} S$ is semismooth 
	as a function from $L^2(\Omega)$ to $L^2(\Omega)$ 
	(in the classical sense of \cite[Definition~3.1]{Ulbrich2011})
	by
	\cref{cor:VIsemismooth1} and 
	since $P, P^* \in \LL(W^*, W)$.
	By using this semismoothness and the uniformly bounded invertibility in \cref{prop:unif_bdd_inv},
	the local q-superlinear convergence of $\{y_i\}$ in $L^2(\Omega)$ follows from standard arguments,
	see, e.g.,
	\cite[Proof of Theorem~3.4]{ChenNashedQi2000}, \cite[Proof of Theorem~3.13]{Ulbrich2011}.
	The claims for $\{u_i\}$ and $\{\tilde y_i\}$ 
	are obtained from the definitions of these sequences, the convergence of $\{y_i\}$, 
	\eqref{eq:optimalityconditionshort},
	\eqref{eq:SLipschitz}, and the continuity of $P$. 
\end{proof}

\section{Numerical experiments for a special instance of problem \texorpdfstring{\eqref{eq:O}}{(OC)}}
\label{sec:6}
To demonstrate that the superlinear convergence predicted by \cref{th:NewtonConvergence}
can also be observed in practice, we present some numerical experiments.
As a model problem, we consider a special instance of \eqref{eq:O}, namely, 
\begin{equation}
\label{eq:M}
\tag{\textup{M}}
	\left \{\,\,
	\begin{aligned}
		\text{Minimize} \quad & \frac12 \norm*{ y - y_D}_{L^2(\Omega)}^2 + 
		\frac{\alpha}{2}\int_\Omega |\nabla u|^2 \, \d x\\
		\text{w.r.t.}\quad &u \in H_0^1(\Omega),\quad y \in H^1(\Omega),\\
		\text{s.t.}\quad & \mathopen{}-\Delta y + y = u \text{ in }\Omega,~~\partial_n y = 0 \text{ on }\partial \Omega,\\
		\text{and}\quad&  u \geq \psi \text{ a.e.\ in }\Omega.
	\end{aligned}
	\right.
\end{equation}
Here and in what follows,
$\Omega$ is assumed to be the unit square, i.e., $\Omega := (0, 1)^2$, 
equipped with the Lebesgue measure, 
$y_D \in C(\overline{\Omega})$ is a given desired state, 
$\alpha > 0$ is a given Tikhonov parameter,
$|\cdot|$ denotes the Euclidean norm, $\nabla$ is the weak gradient, $\Delta$ is the 
distributional Laplacian, $\partial_n$ denotes the normal derivative,  
$H_0^1(\Omega)$ and $H^1(\Omega)$ are defined as usual, see \cite{Attouch2006},
the governing PDE is understood in the weak sense, i.e., in the sense that 
$(y,v)_{H^1(\Omega)} = (u, v)_{L^2(\Omega)}$ holds for all $v \in H^1(\Omega)$,
and 
$\psi$ is a given function satisfying $\psi \in H^2(\Omega)\subset C(\overline{\Omega})$
and $\psi < 0$ on $\partial \Omega$.

It is easy to check that the problem \eqref{eq:M} indeed fits into the general framework of \cref{sec:5} 
with
$\mu$ as the two-dimensional Lebesgue measure, 
$V := H_0^1(\Omega)$, $W := H^1(\Omega)$, $A := - \Delta \in \LL(H_0^1(\Omega), H^{-1}(\Omega))$,
 $Lw  := (w, \cdot)_{H^1(\Omega)}\in H^1(\Omega)^*$ for all $w \in W$,
and $K := \set{v \in V \given v \geq \psi \text{ a.e.\ in }\Omega }$, cf.\ \cref{sec:4}.
In particular, the map $P\colon W^* \to W$ is nothing else than the 
Riesz isomorphism in $H^1(\Omega)$ in the situation of \eqref{eq:M}, i.e., 
\begin{equation}
\label{eq:MPmapping}
\innerp*{ Pz}{ v }_{H^1(\Omega)} = \dual*{ z}{ v }_{H^1(\Omega)}\qquad \forall v \in H^1(\Omega),
z \in H^1(\Omega)^*
\end{equation}
and the solution operator $S\colon H^{-1}(\Omega) \to H_0^1(\Omega)$ of the VI 
\eqref{eq:VI22} is nothing else 
than the solution mapping $z \mapsto w$ of the classical obstacle problem 
\begin{equation}
\label{eq:Mobstacle}
w \in K,\qquad 
\dual*{ -\Delta w - z}{ v - w }_{H_0^1(\Omega)} \geq 0\quad \forall v \in K
\end{equation}
that we have already considered in \cref{subsec:obstacle}.
 Note that the latter implies, in combination with the convexity of $\Omega$, 
the fact that the spatial dimension in \eqref{eq:M} is two, our assumptions on $\psi$,
and the comments at the end of \cref{subsec:obstacle},
that $S(z) \in C(\overline{\Omega})$ holds for all $z \in L^2(\Omega)$
and that the explicit formula for the strong-weak Bouligand differential $\partial_B^{sw}S(z)$
from \cref{th:RaulsWachsmuth} 
is applicable at all points $z \in L^2(\Omega)$. 
This representation formula
allows us to replace \cref{algo1:step:7,algo1:step:8}
in \cref{alg:semiNewton} with the following, more explicit steps
when we apply this algorithm to solve \eqref{eq:M} 
(with the solution operators $P = P^*$ and $S$ of \eqref{eq:MPmapping} and \eqref{eq:Mobstacle}, respectively, and
$\MM_0(\Omega)$, $I(\cdot)$, $A_s(\cdot)$, and $G_\nu$ as in \cref{th:RaulsWachsmuth}):\pagebreak
 \vspace{0.2cm}
\begin{algorithmic}[1]
    \setcounter{ALC@line}{6}
        \STATE 
        Choose $\nu_i \in \MM_0(\Omega)$ satisfying $\nu_i(I(\zeta_i)) = 0$ and $\nu_i = +\infty$ on $A_s(\zeta_i)$. 
        \STATE Determine $y_{i+1} \in L^2(\Omega)$ by solving the linear equation 
        \begin{equation*}
        	y_{i + 1} + \alpha^{-1} P G_{\nu_i}  P y_{i+1} = \tilde y_i + \alpha^{-1} P G_{\nu_i}  P y_i.
        \end{equation*}
\end{algorithmic}
\vspace{0.3cm}
 
To discretize \eqref{eq:M} and to obtain a finite-dimensional counterpart of \cref{alg:semiNewton}, 
we consider standard piecewise affine finite element functions.
Suppose that a family of triangulations $\{\mathcal{T}_h\}_{0<h\leq h_0}$ of the unit square $\Omega = (0,1)^2$
is given (in the sense of \cite[section~II-2.5]{Glowinski2008}). 
We define 
\[
W_h := \set{ v \in C(\overline{\Omega}) \given v |_{T} \text{ is affine for all } T \in \mathcal{T}_h }
\quad
\text{and}
\quad
V_h := W_h \cap H_0^1(\Omega)
\]
and denote with $\{x_k^h\}$ the set of nodes of $\mathcal{T}_h$ and with $I_h\colon C(\overline{\Omega}) \to W_h$
the nodal interpolation operator associated with $W_h$.
By replacing the spaces $V = H_0^1(\Omega)$ and $W = H^1(\Omega)$ in \eqref{eq:M}
with $V_h$ and $W_h$, respectively, by imposing the constraint $u \geq \psi$ only at the mesh nodes $\{x_k^h\}$,
by replacing $y_D$ with its interpolant $I_h(y_D) \in W_h$, 
and by employing a standard discretization of the governing PDE, 
we obtain a family of discrete optimal control problems of the form 
\begin{equation}
\label{eq:Mh}
\tag{\textup{M$_h$}}
	\left \{\,\,
	\begin{aligned}
		\text{Minimize} \quad & \frac12 \norm*{ y_h - I_h(y_D)}_{L^2(\Omega)}^2 + 
		\frac{\alpha}{2}\int_\Omega |\nabla u_h|^2 \, \d x\\
		\text{w.r.t.}\quad &u_h \in V_h,\quad y_h \in W_h,\\
		\text{s.t.}\quad & (y_h, w_h)_{H^1(\Omega)} = (u_h, w_h)_{L^2(\Omega)}~\forall w_h \in W_h\\
		\text{and}\quad&  u_h(x_k^h) \geq \psi(x_k^h) \text{ for all nodes } x_k^h.
	\end{aligned}
	\right.
\end{equation}
Completely analogously to the continuous setting, it can be proved that \eqref{eq:Mh}
possesses exactly one solution $\bar u_h \in V_h$  which is uniquely characterized by the system 
\begin{equation}
\label{eq:discretesys}
\bar z_h = P_h(\bar y_h -I_h(y_D)),\quad 
\bar y_h = P_h(\bar u_h),
\quad \bar u_h = S_h\parens*{- \alpha^{-1}\bar z_h}.
\end{equation}
Here, the operators $P_h\colon L^2(\Omega) \to W_h$ and $S_h\colon L^2(\Omega) \to V_h$ are defined by 
\[
P_h(z) \in W_h,\qquad \innerp*{ P_h(z)}{ v_h }_{H^1(\Omega)} =(z, v_h)_{L^2(\Omega)}\qquad \forall v_h \in W_h
\]
and
\begin{equation}
\label{eq:Sgdef}
S_h(z) \in K_h,\qquad 
\int_\Omega \nabla S_h(z) \cdot \nabla (v_h - S_h(z))
-
z (v_h - S_h(z) ) \, \d x \geq 0\qquad \forall v_h \in K_h,
\end{equation}
respectively, where
$K_h := \set{v_h \in V_h \given v_h(x_k^h) \geq \psi(x_k^h) \text{ for all nodes }x_k^h}$ is the set of admissible controls. 
Note that, analogously to \eqref{eq:optimalityconditionveryshort}, we can restate \eqref{eq:discretesys} as
\begin{equation}
\label{eq:discretesystemshort}
\bar y_h - P_h S_h\parens*{\alpha^{-1} P_h (I_h(y_D) - \bar y_h)} = 0.
\end{equation}
This again yields an equation that is amenable to a semismooth Newton method. 
Since semismoothness properties of solution operators of finite-dimensional obstacle-type VIs 
have already been studied in detail in various contributions, 
e.g., \cite[chapters~5 and 6]{OutrataBook1998}, \cite[sections~4.3 and 5.3]{Bartels2015},
and \cite[section 5.1]{ChristofTR2020}, we omit discussing the 
derivation of Newton derivatives for the map $S_h$ in this paper and simply state the algorithm that 
is obtained by treating the equation \eqref{eq:discretesystemshort} in exactly the same manner as its 
continuous counterpart \eqref{eq:optimalityconditionveryshort}.\pagebreak

 \begin{Algorithm}[Semismooth Newton method for the solution of \eqref{eq:Mh}]\label{alg:semiNewtonMh}
~\hspace{-10cm}
\begin{algorithmic}[1]
  \STATE Choose an initial guess $y_h^0 \in W_h$ and a tolerance $\texttt{tol} \geq 0$.
    \FOR{$i = 0,1,2,3,...$}
    \STATE Calculate $\zeta_h^i :=(P_h I_h(y_D) - P_h y_h^i)/\alpha$, $u_h^i := S_h(\zeta_h^i)$, and $\tilde y_h^i := P_h u_h^i$.
        \IF{$\|y_h^i - \tilde y_h^i\|_{L^2(\Omega)} \leq \texttt{tol}$}
            \STATE STOP the iteration (convergence reached). 
        \ELSE
        \STATE 
        Choose a subset $\NN_i$ of the set of nodes $\{x_k^h\}$ of $\mathcal{T}_h$ 
        that contains all strictly active nodes 
        of $S_h(\zeta_h^i)$ and none of the inactive nodes of $S_h(\zeta_h^i)$.
        \STATE Determine $y_h^{i+1} \in W_h$ by solving the linear equation 
        \begin{equation}
        \label{eq:discreteNewtonupdate}
        	y_h^{i + 1} + \alpha^{-1} P_h G_{\NN_i}  P_h y_h^{i+1} = \tilde y_h^i + \alpha^{-1} P_h G_{\NN_i}  P_h y_h^i,
        \end{equation}
        where $G_{\NN_i}$ denotes the solution map $L^2(\Omega) \ni z \mapsto w_h \in V_h$ of the problem 
        \begin{equation*}
w_h \in Z_{\NN_i},\qquad 
\int_\Omega \nabla w_h \cdot \nabla  v_h \, \d x 
=
\int_\Omega  z  v_h  \, \d x \quad \forall v_h \in Z_{\NN_i}
\end{equation*}
with 
\[
Z_{\NN_i} := \set{v_h \in V_h \given v_h(x_k^h) = 0 \text{ for all }x_k^h \in \NN_i }.
\]
        \ENDIF
    \ENDFOR
\end{algorithmic}
\end{Algorithm}

Here, the inactive nodes of $S_h(\zeta_h^i)$ are, as usual, defined to be those nodes $x_k^h$ which satisfy
$S_h(\zeta_h^i)(x_k^h) > \psi(x_k^h)$ and the strictly active nodes of $S_h(\zeta_h^i)$
are those nodes  $x_k^h$ at which the (scalar) Lagrange multiplier associated with the 
constraint $S_h(\zeta_h^i)(x_k^h) \geq \psi(x_k^h)$ in \eqref{eq:Sgdef} is nonzero, cf.\ \cite[page~93]{OutrataBook1998}. 
We remark that, in practice, 
the Newton update in \cref{alg:semiNewtonMh} is, of course, not calculated by solving 
\eqref{eq:discreteNewtonupdate} as is. Instead, one rewrites this equation 
as a linear system that involves the mass and stiffness matrices 
associated with $V_h$ and $W_h$ and auxiliary variables that 
decompose the composition $P_h G_{\NN_i}  P_h$ into three sparse linear equations.

The results that we have obtained with \cref{alg:semiNewtonMh} in the situation 
of \eqref{eq:M} (or \eqref{eq:Mh}, respectively) for $\alpha = 10^{-5}$, $y_D(x_1, x_2) := -x_1 - x_2$,
$\psi(x_1, x_2) := -5$, and Friedrichs--Keller triangulations  $\{\mathcal{T}_h\}$ 
with  various widths $h$ can be seen in \cref{tab:1,fig:experimentresults} below. 
In all of the depicted experiments, the initial guess $y_h^0$ was chosen as $I_h(y_D)$,
the tolerance for the semismooth Newton method was $\texttt{tol} = 10^{-7}$,
the set $\NN_i$ was chosen as the set of strictly active nodes for all $i$
(with strictly active defined up to the tolerance $10^{-10}$), 
and the linear systems of equations arising in \cref{alg:semiNewtonMh}  and the discrete obstacle problem 
\eqref{eq:Sgdef} were solved with Matlab2020b's backslash solver and 
quadprog-routine, respectively. For the calculation of the experimental orders of convergence (EOCs)
in  \cref{tab:1}, we used the formula 
\begin{equation}
\label{eq:EOC}
\textup{EOC}_i := 
\log\parens[\bigg]{ \frac{\norm{ v_h^{i}-v_h^{i-1} }}{\norm{ v_h^{i-1}-v_h^{i-2} }} } / \log\parens[\bigg]{ \frac{\norm{ v_h^{i-1}-v_h^{i-2} }}{\norm{ v_h^{i-2}-v_h^{i-3} }} }
\end{equation}
for $i=6$, i.e., for the largest $i$ reached in all numerical experiments. 

\begin{table}[ht]
\caption{Number of performed Newton iterations, final residue $\smash{\|y_h^i - \tilde y_h^i\|_{L^2(\Omega)}}$, 
and experimental orders of convergence (EOCs) for the iterates $\{y_h^i\}$ in $L^2(\Omega)$,
$\{\tilde y_h^i\}$ in $H^1(\Omega)$, and $\{u_h^i\}$ in $H_0^1(\Omega)$
obtained from \cref{alg:semiNewtonMh} for $\alpha$, $y_D$, $\psi$, etc.\
as described above and various mesh widths $h$.}
\label{tab:1}
\centering
\begin{tabular}{c | c | c | c | c | c}
\hline\noalign{\smallskip}
$h$ &  iter. & fin.\ res.\ & $L^2$-EOC $y_h^i$ &  $H^1$-EOC $\tilde y_h^i$ & $H_0^1$-EOC $u_h^i$ \\
\noalign{\smallskip}\hline\noalign{\smallskip}
$\frac{1}{16}$  & $6$ & $7.3259 \cdot 10^{-11}$ & $1.7560$ &  $2.0441$ & $2.1076$   \\[0.1cm]
$\frac{1}{32}$  & $6$ & $3.1726 \cdot 10^{-10}$ & $1.5567 $ &  $1.8209$ & $1.7154$   \\[0.1cm]
$\frac{1}{64}$  & $6$ & $1.0445 \cdot 10^{-9}$ & $1.7783$ &  $2.0074$ & $1.9553$   \\[0.1cm]
$\frac{1}{128}$ & $7$ & $2.7176 \cdot 10^{-9}$ & $1.7344$ &  $2.0158$ & $1.9445$   \\[0.1cm]
$\frac{1}{256}$ & $6$ & $4.1500 \cdot 10^{-8}$ & $1.8745$ &  $2.1046$ & $2.0719$   \\[0.1cm]
$\frac{1}{512}$ & $6$ & $3.7828\cdot 10^{-8}$ & $1.8744$ &  $2.1047$ & $2.0760$   \\[0.1cm]
\noalign{\smallskip}\hline
\end{tabular}%
\end{table}

\begin{figure}[ht]
\centering
\subfigure[Desired state $y_D$]{\includegraphics[width=5.8cm,height=7cm,keepaspectratio]{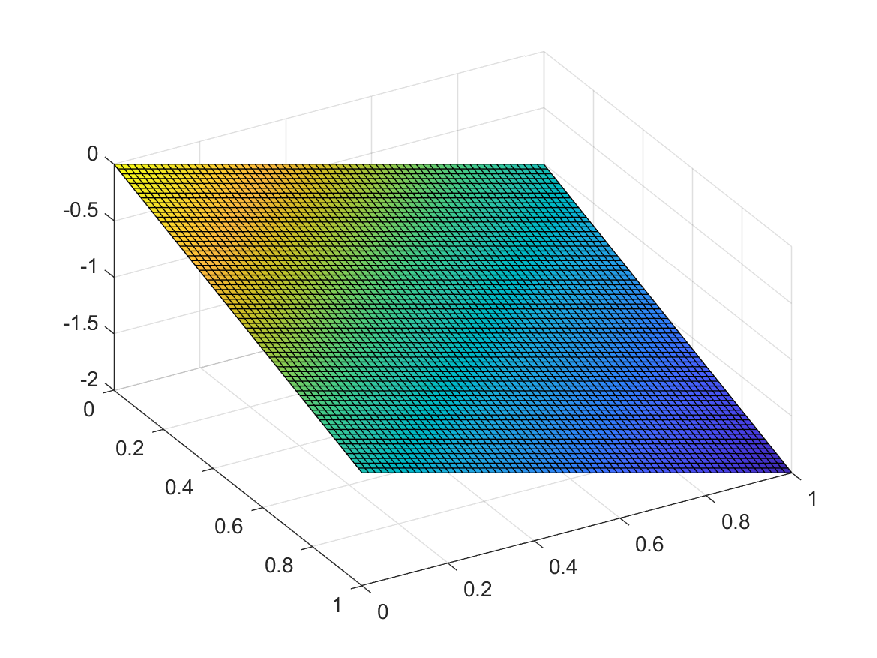}} %
\hspace{0.5cm}%
\subfigure[Final state $\tilde y_h^{i}$]{\includegraphics[width=5.8cm,height=7cm,keepaspectratio]{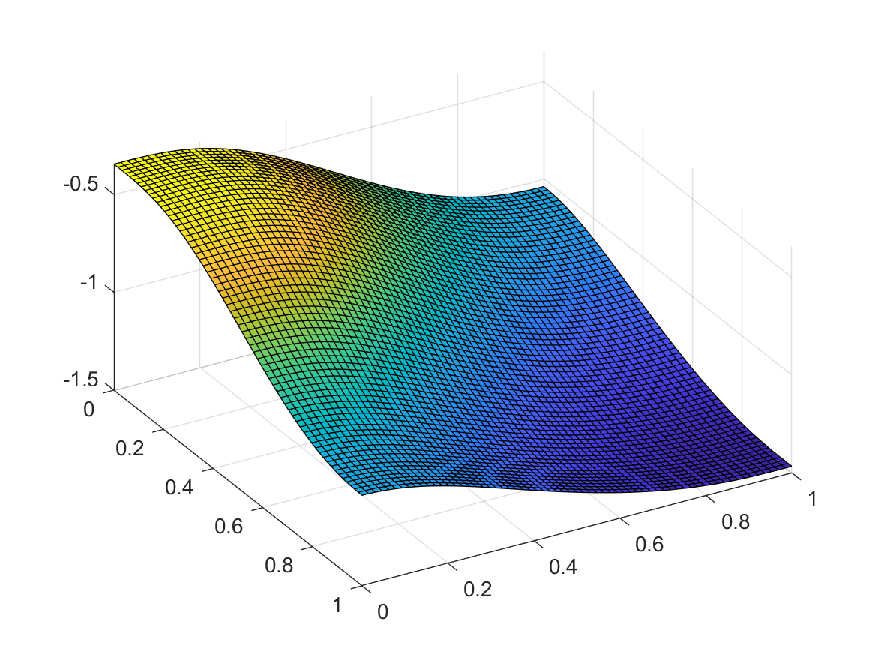}} %
\\%
\subfigure[Final control $u_h^i$]{\includegraphics[width=5.8cm,height=7cm,keepaspectratio]{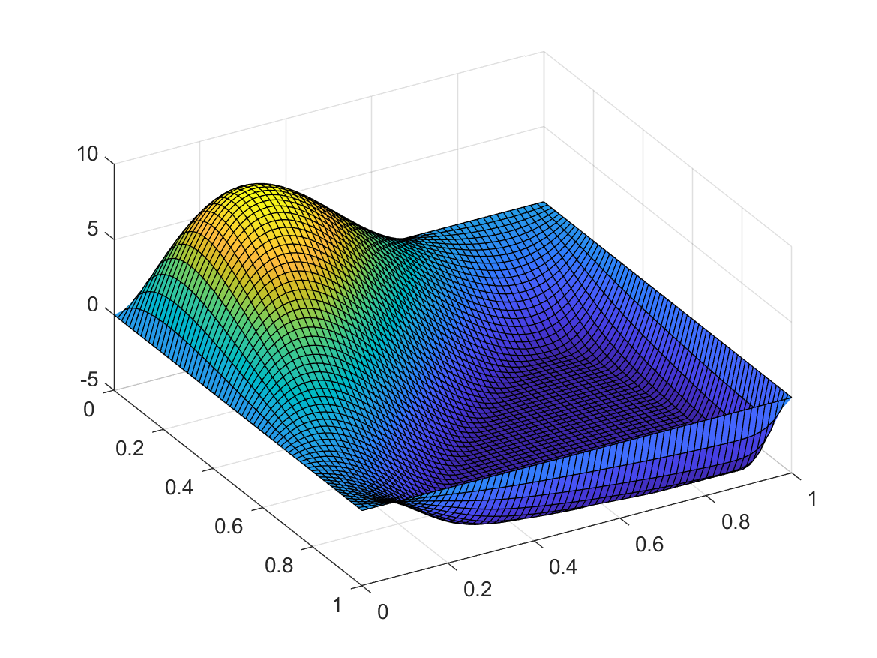}} %
\hspace{0.5cm}%
\subfigure[Final multiplier of $u_h^i$]{\includegraphics[width=5.8cm,height=7cm,keepaspectratio]{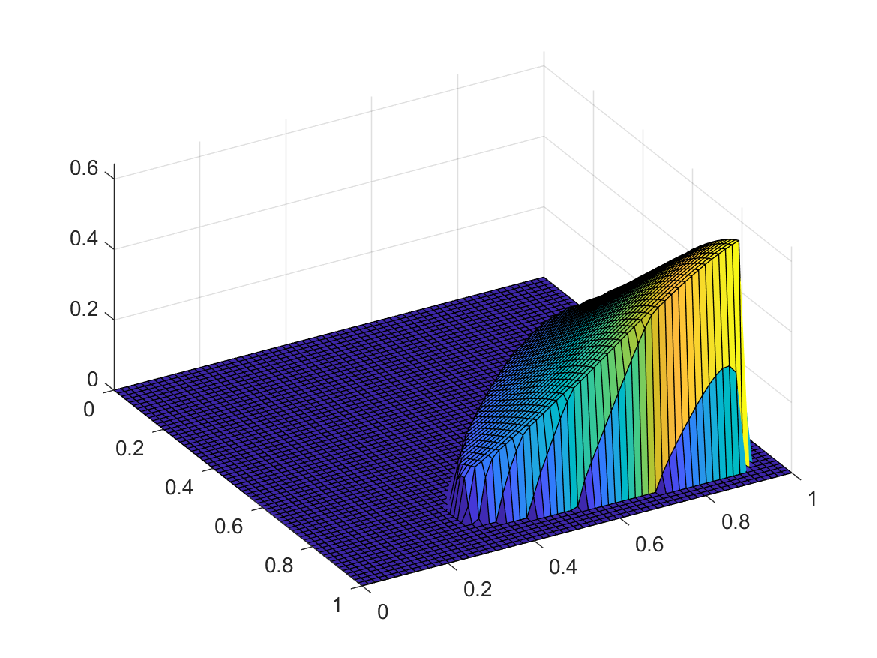}} %
\caption{Numerical results obtained with \cref{alg:semiNewtonMh} for $\alpha$, $y_D$, $\psi$, etc.\ as
described above and $h = 1/64$. The images show the desired state, 
the final iterate $\tilde y_h^i$, the final iterate $u_h^i$, and the Lagrange multiplier of 
$u_h^i = S_h(\zeta_h^i)$ in \eqref{eq:Sgdef}, respectively. Convergence was reached in this test case 
in six iterations with the final residue $\|y_h^i - \tilde y_h^i\|_{L^2(\Omega)} \approx 10^{-9}$, see \cref{tab:1}.}
\label{fig:experimentresults}
\end{figure}

As \cref{tab:1} shows,
\cref{alg:semiNewtonMh} indeed converges mesh-independently, with the 
number of iterations necessary for getting the residue below the tolerance $\texttt{tol} = 10^{-7}$
being nearly constant at six. It can also be observed that the experimental orders of convergence 
obtained from the approximation formula \eqref{eq:EOC} for
$\{y_h^i\}$ in $L^2(\Omega)$,
$\{\tilde y_h^i\}$ in $H^1(\Omega)$, and $\{u_h^i\}$ in $H_0^1(\Omega)$
seem to converge for $h \to 0$. This behavior for vanishing $h$ is the main motivation for 
studying the convergence of \cref{alg:semiNewtonMh} in the function space setting. 
Note that \cref{tab:1} indicates that the order of convergence of the sequence $\{\tilde y_h^i\}$
in $H^1(\Omega)$ is significantly higher than that of the sequence $\{y_h^i\}$ in $L^2(\Omega)$
(around two compared to approximately $1.8$). Whether this effect has roots in some analytical 
properties of \eqref{eq:M} and whether the sequences $\{\tilde y_h^i\}$ and $\{u_h^i\}$
converge not only r-superlinearly but even q-quadratically (as suggested by the last two columns of \cref{tab:1})
is currently unclear. We leave this question for further research. 

We conclude this paper by pointing out two further possible applications 
of the Newton differentiability results that we have established for solution operators 
of VIs with unilateral constraints in \cref{sec:3}. 

First, we expect that \cref{cor:VIsemismooth1} is also helpful for the 
study of optimization algorithms for optimal control problems that are governed 
by obstacle-type VIs. In the finite-dimensional setting, 
bundle-type methods, for example, are often globalized by means 
of a line-search that requires the objective function to be semismooth, see
\cite{Lemarechal1981-2,SchrammZowe1992}. 
With \cref{cor:VIsemismooth1} at hand, which immediately 
yields that the reduced objective function of, for instance, a tracking-type 
optimal control problem for the classical obstacle problem is semismooth, 
it may be possible to use similar techniques in the infinite-dimensional setting, cf.\ \cite{HertleinRaulsUlbrichUlbrich2022,HertleinUlbrich2018}.

A second potential application area for \cref{cor:VIsemismooth1} is the development 
of solution algorithms for obstacle-type quasi-VIs, i.e., 
problems of the form \eqref{eq:classobstacle} in which the obstacle $\psi$ depends implicitly 
on the solution $y$, cf.\ \cite{AlphonseHintermuellerRautenberg2019,ChristofWachsmuth2021,Wachsmuth2020} and the references therein.
For sufficiently regular functions $y \mapsto \psi(y)$,
such problems can be written in the form of a fixed-point equation that involves the solution map $S$
of an obstacle-type VI as studied in \cref{sec:3}. 
Using \cref{cor:VIsemismooth1}, it may be possible to set up a
semismooth Newton method for this fixed-point equation and to thus develop 
numerical solution algorithms whose convergence can be established in function space. 
We remark that, in the finite-dimensional setting, such techniques have already been used, see \cite{Xie2021}. 
We leave both of these topics for future research.


\bibliographystyle{siamplain}
\bibliography{references}
\end{document}